\documentclass[11pt]{amsart}

\usepackage[T1]{fontenc}
\usepackage{lmodern}
\usepackage{microtype}
\usepackage[margin=0.9in]{geometry}
\usepackage[authoryear,round,longnamesfirst]{natbib}
\usepackage[shortlabels]{enumitem}
\usepackage{mathtools,amssymb,thmtools}
\usepackage{mathrsfs}
\usepackage[dvipsnames]{xcolor}
\usepackage[
  colorlinks=true,
  linkcolor=MidnightBlue,
  citecolor=OliveGreen,
  urlcolor=BrickRed
]{hyperref}
\usepackage[noabbrev,capitalise,nameinlink]{cleveref}

\allowdisplaybreaks

\newcommand{\Dm}{\mathcal{D}_m}
\newcommand{\Dmt}{\widetilde{\mathcal{D}}_m}
\newcommand{\HH}{\mathcal{H}}
\newcommand{\CC}{\mathbb{C}}
\newcommand{\RR}{\mathbb{R}}
\newcommand{\ZZ}{\mathbb{Z}}

\newcommand{\Ep}{E^+}
\newcommand{\Em}{E^-}

\newcommand{\sgn}{\mathrm{sgn}}
\newcommand{\id}{\mathrm{id}}

\newcommand{\Sphere}{\mathbb{S}}
\newcommand{\pauli}{\sigma}
\newcommand{\norm}[1]{\left\|#1\right\|}
\newcommand{\abs}[1]{\left|#1\right|}
\newcommand{\inner}[2]{\left\langle #1,\, #2 \right\rangle}
\newcommand{\dx}{\,\mathrm{d}x}

\newcommand{\dt}{\,\mathrm{d}t}

\newcommand{\dy}{\,\mathrm{d}y}
\DeclareMathOperator{\RE}{Re}
\newcommand{\dom}{\operatorname{dom}}
\newcommand{\PiPW}{\Pi_{\mathrm{pw}}}
\newtheorem{theorem}{Theorem}
\newtheorem{lemma}[theorem]{Lemma}
\newtheorem{proposition}[theorem]{Proposition}

\theoremstyle{definition}
\newtheorem{definition}[theorem]{Definition}
\newtheorem{example}[theorem]{Example}
\theoremstyle{remark}
\newtheorem{remark}[theorem]{Remark}

\title[Semiconvex perturbations of Dirac equations]
{Perturbation from symmetry for semiconvex lower-order partial-wave nonlinear Dirac equations}

\author{Francesco Paolo Maiale}
\address{Gran Sasso Science Institute, Viale Luigi Rendina 26--28,
67100 L'Aquila, Italy}
\email{francescopaolo.maiale@gssi.it}
\urladdr{https://orcid.org/0000-0002-3389-6937}

\date{\today}

\hypersetup{
  pdftitle={Perturbation from symmetry for semiconvex lower-order partial-wave nonlinear Dirac equations},
  pdfauthor={Francesco Paolo Maiale},
  pdfkeywords={Dirac functional, partial waves, perturbation from symmetry, strongly indefinite functional, lower-order perturbation}
}

\begin{document}
\begin{abstract}
We study strongly indefinite nonlinear Dirac functionals in the lowest partial-wave channel with radial coefficients in $\RR^3$. The symmetric model has the exact power nonlinearity $F_0(x,\psi)=b(|x|)|\psi|^p/p$, where $2<p<3$ and the radial profile $b$ is bounded, continuous, and uniformly positive. Compactness in this finite-angular-mode spinor channel and radial approximation-number estimates give quadratic growth of the symmetric minimax levels.

For a non-even localized perturbation of order $1<\tau<p/2$, we impose a semiconvexity condition whose negative curvature is strictly smaller than the spectral gap. Together with the negative quadratic part and the convex exact-power core, this makes every negative spectral fiber uniformly strongly concave. Maximizing along that fiber reduces the problem to a $C^1$ path on the positive spectral space. The polynomial Chambers--Ghoussoub--Bolle deformation theorem then yields infinitely many high-energy critical points of the restricted partial-wave functional; under the channel-invariance hypothesis these are weak solutions of the full Dirac equation. Convex perturbing primitives are covered without a smallness restriction on their amplitude.

The same reduction applies after an even Hermitian quadratic term is absorbed into a renormalized Dirac operator, provided the renormalized operator has a gap at zero and the remaining perturbation satisfies the corresponding semiconvexity bound.
\end{abstract}
\maketitle

% =====================================================================
\section{Introduction}\label{sec:intro}
% =====================================================================

We study stationary nonlinear Dirac equations of the form
\begin{equation}\label{eq:main}
  \Dm\psi(x)+V(x)\psi(x)=f(x,\psi(x)),\qquad x\in\RR^3,
\end{equation}
through their variational formulation on $H^{1/2}(\RR^3;\CC^4)$. The Dirac operator with mass $m>0$ has positive and negative spectral branches separated
by a gap, so the associated quadratic form is unbounded both from above and
from below on infinite-dimensional subspaces. This strong indefiniteness is
the first issue. The second is the loss of compactness on the
whole space, since the subcritical embeddings of $H^{1/2}(\RR^3;\CC^4)$ are
continuous but not compact. These two features distinguish nonlinear Dirac
problems from their semibounded Schr\"odinger counterparts; see
\cite{Thaller1992,EstebanSere1995,EstebanLewinSere2008} for the operator and
variational background.

Variational methods for stationary nonlinear Dirac equations go back, in
particular, to the linking construction of \citet{EstebanSere1995}. The
subsequent theory includes general strongly indefinite critical-point
methods, multiplicity results, and semiclassical concentration phenomena;
see, for instance,
\cite{BenciRabinowitz1979,KryszewskiSzulkin1998,BartschDingDirac2006,DingXu2014,BartschXu2020}.
When the nonlinear primitive is even, the natural $\ZZ_2$ symmetry makes
genus, fountain, and symmetric minimax constructions available.  The
question considered here is whether an infinite high-energy family survives
a genuinely non-even lower-order perturbation.

Our compactness mechanism is angular rather than confining. Instead of
introducing an external confining potential, we restrict the functional to one
partial-wave channel of the three-dimensional Dirac operator. In spherical
variables $x=r\omega$, a four-component spinor is viewed as an upper and a
lower two-spinor. The angular dependence of these two-spinors is described by
spinor spherical harmonics
\[
  \Omega_{\kappa\mu}:\mathbb S^2\to\CC^2,
\]
which play for the Dirac operator the role played by the scalar spherical
harmonics $Y_{\ell m}$ for the Laplacian: they diagonalize the angular
momentum and reduce radial scalar Dirac operators to radial systems in each
channel. The lowest channel used here consists of spinors of the form
\[
  \psi(r\omega)
  =
  \begin{pmatrix}
    u(r)\, \Omega_{-1,1/2}(\omega)\\
    \imath v(r) \, \Omega_{1,1/2}(\omega)
  \end{pmatrix},
\]
where $u$ and $v$ are radial coefficients and the two $\Omega$'s are fixed
normalized angular spinors. Thus the space is not the scalar radial subspace
of $H^{1/2}(\RR^3;\CC^4)$; rather, it is a finite-angular-mode spinor
subspace with two radial profiles.

This fixed channel still recovers the compactness mechanism associated with
radial symmetry. The special lowest-mode harmonics have constant angular
density, which gives
\[
  |\psi(r\omega)|^2=\tfrac1{4\pi} \bigl(|u(r)|^2+|v(r)|^2\bigr).
\]
Consequently radial exact-power nonlinearities preserve the same channel, and
the radial coefficients enjoy compact subcritical embeddings. The underlying
principle is the classical compactness recovered by radial symmetry in
\cite{Strauss1977,Lions1982}, together with its refinements in radial Besov,
Triebel--Lizorkin, and fractional Sobolev spaces of
\cite{SickelSkrzypczak2000,SickelSkrzypczakVybiral2012}. The partial-wave
decomposition itself is standard in Dirac spectral theory (see, e.g.,~\cite{Thaller1992}). For the present multiplicity argument, compactness alone
is not enough: we also use the approximation-number asymptotics of
\cite{DotaSkrzypczak2025}. After transferring those estimates through the
fixed angular spinor modes, they yield the polynomial tail bound that is
responsible for quadratic growth of the symmetric minimax levels.

The loss of evenness has a long history in critical-point theory. Early
perturbation-from-symmetry methods were developed by
\cite{Struwe1980,BahriBerestycki1981,Rabinowitz1982}; the deformation theory
was subsequently refined by \cite{Bolle1999,BolleGhoussoubTehrani2000,ChambersGhoussoub2001}. We use the convenient $C^1$ polynomial-growth
formulation recorded by \cite{SalvatoreSquassina2003}.  Related
perturbation-from-symmetry results for strongly indefinite elliptic systems
were obtained by \cite{Tarsi2006}. A direct application of such a theorem to
a Dirac functional is nevertheless not automatic: the infinite-dimensional
negative spectral space remains present throughout the deformation and must
be controlled uniformly.

The aim of this work is to extend the results obtained in \citet{GeorgievMaiale2018}, which proved the
existence of a nonzero radially symmetric critical point for a Dirac
functional with broken symmetry. The present work, however, addresses a different
multiplicity question: under a structured lower-order symmetry breaking, can
one retain an unbounded sequence of critical values? Our assumptions and
those of \cite{GeorgievMaiale2018} are not nested, so the result below is best
viewed as complementary rather than as a formal strengthening in every
respect. Its new feature is the persistence of infinitely many high-energy
partial-wave states.

The main ingredient is a global reduction along the negative spectral fibers.
For every parameter $\vartheta\in[0,1]$ and positive spectral component
$x\in X^+$, we maximize the energy over $x+X^-$. The negative quadratic
part, the convex exact-power core, and a semiconvexity bound on the perturbing
primitive, with negative curvature strictly below the spectral gap, make every
such fiber uniformly strongly concave. Hence the maximizer
$h_\vartheta(x)\in X^-$ is unique, depends continuously on $(\vartheta,x)$,
and defines a jointly $C^1$ reduced family
\[
  J_\vartheta(x) = E_\vartheta\bigl(x+h_\vartheta(x)\bigr)
  \qquad\text{on }X^+.
\]
Critical points and critical values of the reduced family correspond exactly
to critical points and critical values of the functional restricted to the
chosen partial-wave space; the channel-invariance hypothesis \textnormal{(X3)} then upgrades
these to full weak solutions. In this way the strongly indefinite problem is
converted into a deformation problem on the positive Hilbert space, where the
polynomial perturbation-from-symmetry theorem applies.

The proof combines four ingredients. First, the fiber reduction
preserves the joint Palais--Smale property and gives uniform
finite-dimensional anti-coercivity. Second, the lower-order perturbation
produces a critical velocity bounded by $C(1+|c|^{\tau/p})$, hence
asymptotically of order $|c|^{\tau/p}$. Third, radial
approximation numbers give
\[
  c_k\ge C_1k^2-C_2
\]
for the symmetric reduced minimax levels. Finally, the deformation theorem
compares this quadratic growth with the scalar-flow exponent
$1/(1-\tau/p)$. The strict inequality
\[
  2>\tfrac1{1-\tau/p}
  \iff
  \tau<\tfrac p2
\]
is therefore the natural threshold of this method.

\subsection*{Main result and scope}

Let
\[
  F_0(x,\psi)=\tfrac{b(|x|)}p|\psi|^p,
  \qquad b\in C_b([0,\infty);\RR),
  \quad 0<b_0\le b(r)\le b_1<\infty,
  \quad 2<p<3,
\]
and let $G(x,\psi)$ be a localized Carath\'eodory primitive, $C^1$ in the
spinor variable, of order $1<\tau<p/2$, not necessarily even. Besides the weighted growth conditions,
assume that for some non-negative $\mu\in L^\infty(\RR^3)$,
\[
 \RE\langle g(x,z)-g(x,w),z-w\rangle
 \ge-\mu(x)|z-w|^2,
 \qquad \|\mu\|_\infty<\lambda_1,
\]
where $g=\nabla_\psi^{\RR}G$ and $\lambda_1$ is the spectral-gap constant.
Assume also that the scalar potential and the partial-wave space satisfy the
spectral-gap and channel-compatibility hypotheses stated in~\cref{sec:assumptions}. Then the restricted
partial-wave functional has infinitely many critical points whose energies
tend to $+\infty$; see \cref{thm:partialwave_lower_order}. Under the
channel-invariance condition, these restricted critical points are full weak
solutions. When $G(x,\cdot)$ is convex the condition holds with $\mu=0$, so
no smallness is imposed on the amplitude of the convex perturbation.

The hypotheses also indicate the present scope. We work in the lowest
partial-wave channel with radial coefficients, use an
exact-power even core, and require the non-even
term to be both lower order and semiconvex relative to the spectral gap. We
do not claim persistence for arbitrary non-even perturbations, for which the
negative-fiber maximizer may cease to be unique. An even Hermitian quadratic
term can, however, be absorbed into a renormalized Dirac operator. The same
conclusion then holds whenever the renormalized operator has a gap at zero and
the remaining perturbation satisfies the corresponding semiconvexity bound;
see \cref{thm:renormalized_even_quadratic}.

\subsection*{Organization of the paper}

\cref{sec:setting} introduces the Dirac splitting and energy space.
\cref{sec:partialwaves} constructs the lowest partial-wave compactness space
and transfers radial approximation estimates to the spinor channel.
\cref{sec:assumptions} states the nonlinear hypotheses, including
semiconvexity of the perturbation.  \cref{sec:symmetric} proves the symmetric
fountain theorem.  \cref{sec:perturbation_symmetry} develops the
negative-fiber reduction, verifies the deformation velocities, and proves
quadratic growth of the reduced symmetric levels.
\cref{sec:lower_order_reduced} proves the non-even lower-order theorem, and
\cref{sec:renormalized_quadratic} treats even quadratic terms by renormalizing
the Dirac operator.

% =====================================================================

\section{Functional setting}\label{sec:setting}
% =====================================================================

The Dirac operator with mass $m > 0$ on $\RR^N$, for $N \geq 3$, acts on spinor-valued functions $\psi : \RR^N \to \CC^d$ and is given by
\[
  \Dm \coloneqq -\imath \sum_{j=1}^{N} \alpha^j\partial_j + m\beta, \qquad d \coloneqq 2^{\lfloor (N+1)/2 \rfloor}.
\]
The matrices $\alpha^1,\ldots,\alpha^N,\beta\in\CC^{d\times d}$ are Hermitian and satisfy the following anti-commutator relations:
\begin{equation}\label{eq:clifford_relations}
  \alpha^j\alpha^k+\alpha^k\alpha^j=2\delta_{jk}I_d, \qquad \alpha^j\beta+\beta\alpha^j=0, \qquad  \beta^2=I_d.
\end{equation}
With these conventions $\Dm$ is self-adjoint on $L^2(\RR^N;\CC^d)$, its domain is $H^1(\RR^N;\CC^d)$, and its spectrum is
\[
  \sigma(\Dm) = (-\infty,-m] \cup [m,+\infty).
\]
Let $V : \RR^N \to \RR$ be a real-valued scalar potential with
$V\in L^\infty(\RR^N)$, acting on spinors as the multiplication operator
$V(x) I_d$, and define
\[
  \Dmt \coloneqq \Dm + V(x) I_d.
\]
Since $V \in L^\infty(\RR^N)$, $\Dmt$ is a bounded
self-adjoint perturbation of $\Dm$. Hence $\Dmt$ is self-adjoint on
$L^2(\RR^N;\CC^d)$ with domain $H^1(\RR^N;\CC^d)$, and the graph norms of
$\Dm$ and $\Dmt$ are equivalent. Therefore the interpolation spaces
\[
  \bigl[L^2(\RR^N;\CC^d),\dom(\Dm)\bigr]_{1/2}
  \quad\text{and}\quad
  \bigl[L^2(\RR^N;\CC^d),\dom(\Dmt)\bigr]_{1/2}
\]
coincide with equivalent norms. Since, for a self-adjoint operator $A$,
\[
  \bigl[L^2,\dom(A)\bigr]_{1/2}=\dom\bigl(|A|^{1/2}\bigr)
\]
with equivalent norms, we obtain
\[
 \dom \bigl(|\Dmt|^{1/2}\bigr)=H^{1/2}(\RR^N;\CC^d)
\]
with equivalence of the full form graph norm and the standard
$H^{1/2}$-norm. In the abstract formulation below this property is recorded
as assumption~\textnormal{(V2)}.

Throughout the remainder of this section we assume the following \textbf{spectral-gap condition}:
\begin{equation}\label{eq:lambda1_gap_setting}
  \lambda_1\coloneqq\operatorname{dist}(0,\sigma(\Dmt))>0.
\end{equation}
It ensures that $\sgn(\Dmt)^2=\id$, so that the operators $P_m$ and $Q_m$
defined below in~\eqref{eq:PmQm} are genuine idempotents. It also bounds
$|\Dmt|$ away from zero, so the homogeneous form norm used below is equivalent
to the full form graph norm and hence to the standard $H^{1/2}$-norm. We then define the spectral projections by functional calculus:
\begin{equation}\label{eq:PmQm}
  P_m \coloneqq \frac{\id + \sgn(\Dmt)}{2}, \qquad Q_m \coloneqq \frac{\id - \sgn(\Dmt)}{2}.
\end{equation}
Thus $P_m^2 = P_m$, $Q_m^2 = Q_m$, $P_m + Q_m = \id$, and $P_m - Q_m = \sgn(\Dmt)$. Although the notation displays only the mass, $P_m$ and $Q_m$ are always the spectral projections of the perturbed operator $\Dmt$.

\subsection{The energy space}

We work on the fractional Sobolev space $\HH \coloneqq H^{1/2}(\RR^N;\CC^d)$, endowed with the real Hilbert inner product
\[
  (\psi,\varphi)_\HH \coloneqq \RE\inner{\abs{\Dmt}^{1/2}\psi}{\abs{\Dmt}^{1/2}\varphi}_{L^2},
\]
and the induced norm
\[
  \norm{\psi} \coloneqq (\psi,\psi)_\HH^{1/2}.
\]
By the spectral gap assumption and the boundedness of $V$, this norm is equivalent to the standard $H^{1/2}$-norm.

Note that $P_m$, $Q_m$, and $\abs{\Dmt}^{1/2}$ are obtained from the same
self-adjoint operator by functional calculus. Hence $P_m$ and $Q_m$ preserve
$\dom \bigl(|\Dmt|^{1/2}\bigr)=\HH$ and commute with $\abs{\Dmt}^{1/2}$ on
this domain.

\begin{proposition}\label{prop:spectral_split}
The operators $P_m$ and $Q_m$ restrict to continuous projections on $\HH$, and the space decomposes as
\begin{equation}\label{eq:Hdecomp}
  \HH = \Ep \oplus \Em, \qquad \Ep \coloneqq P_m(\HH), \qquad \Em \coloneqq Q_m(\HH).
\end{equation}
Consequently, every $\psi \in \HH$ decomposes uniquely as
\[
  \psi = \psi^+ + \psi^-, \qquad \psi^{\pm} \in E^\pm,
\]
and the decomposition is orthogonal with respect to $(\cdot,\cdot)_\HH$:
\[
  \norm{\psi}^2 = \norm{\psi^+}^2 + \norm{\psi^-}^2.
\]
\end{proposition}

\begin{proof}
Let $\psi \in \HH$. Since $P_m$ commutes with $\abs{\Dmt}^{1/2}$ and is bounded on $L^2$, we have
\[
  \norm{P_m\psi}^2 = \norm{\abs{\Dmt}^{1/2} P_m\psi}_{L^2}^2
   = \norm{P_m \abs{\Dmt}^{1/2}\psi}_{L^2}^2  \le \norm{\abs{\Dmt}^{1/2}\psi}_{L^2}^2 = \norm{\psi}^2.
\]
Hence $P_m$ maps $\HH$ continuously into itself. The same argument applies to the other projection $Q_m$. Now define
\[
  \psi^+ \coloneqq P_m\psi, \qquad \psi^- \coloneqq Q_m\psi.
\]
Since $P_m + Q_m = \id$, we have $\psi = \psi^+ + \psi^-$; moreover, the decomposition is unique because $P_m Q_m = Q_m P_m = 0$.

Finally, since $P_m$, $Q_m$, and $\abs{\Dmt}^{1/2}$ commute, and $P_m Q_m = 0$, we get
\[
  (\psi^+,\psi^-)_\HH = \RE\inner{\abs{\Dmt}^{1/2}P_m\psi}{\abs{\Dmt}^{1/2}Q_m\psi}_{L^2} = \RE\inner{P_m\abs{\Dmt}^{1/2}\psi}{Q_m\abs{\Dmt}^{1/2}\psi}_{L^2} = 0,
\]
which implies $\norm{\psi}^2 = \norm{\psi^+}^2 + \norm{\psi^-}^2$.
\end{proof}

\subsection{The compactness subspace}\label{ssec:compactness}

To recover compactness on the unbounded domain $\RR^N$, we work on a closed subspace $X \subset \HH$ satisfying the following assumptions. First, define the critical exponent
\[
  2^*_{1/2} \coloneqq \frac{2N}{N-1},
\]
as well as the projected spaces:
\[
  X^+ \coloneqq P_m(X), \qquad X^- \coloneqq Q_m(X).
\]
We assume that $X$ satisfies the following two assumptions:
\begin{enumerate}
\item[\textnormal{(X1)}]\label{it:X1}
  $X^\pm \subset X$ and both $X^+$ and $X^-$ are infinite-dimensional.

\item[\textnormal{(X2)}]\label{it:X2}
  For every $q \in (2,2^*_{1/2})$, the embedding $X \hookrightarrow L^q(\RR^N;\CC^d)$ is continuous and compact.
\end{enumerate}

We equip $X$ with the norm induced by $\HH$. Then $X$ also decomposes orthogonally with respect to the induced inner product, as shown by the next proposition.

\begin{proposition}\label{prop:compact}
Assume that $X \subset \HH$ is a closed subspace satisfying
\textnormal{(X1)}. Then
\[
  X^+=X\cap E^+, \qquad X^-=X\cap E^-.
\]
In particular, $X^+$ and $X^-$ are closed subspaces of $X$. Moreover,
\[
  X = X^+ \oplus X^-,
\]
and the decomposition is orthogonal with respect to $(\cdot,\cdot)_\HH$. Hence every $\psi \in X$ decomposes uniquely as
\[
  \psi = \psi^+ + \psi^-, \qquad \psi^\pm \in X^\pm.
\]
\end{proposition}

\begin{proof}
We first identify the projected spaces with intersections. If $u\in X^+$,
then $u=P_mx$ for some $x\in X$. Hence $u\in E^+$, and by
\textnormal{(X1)} also $u\in X$; therefore $u\in X\cap E^+$. Conversely, if
$u\in X\cap E^+$, then $u=P_mu\in P_m(X)=X^+$. Thus $X^+=X\cap E^+$. The
same argument gives $X^-=X\cap E^-$. Since $X$ and $E^\pm$ are closed in
$\HH$, the spaces $X^\pm$ are closed in $X$.

By definition and \textnormal{(X1)}, $X^\pm\subset X$, so
$X^+ + X^-\subset X$. Conversely, let $\psi \in X$. Again by
\textnormal{(X1)}, both $P_m\psi$ and $Q_m\psi$ belong to $X$, and
\[
  \psi = P_m\psi + Q_m\psi \in X^+ + X^-.
\]
Thus $X=X^+ + X^-$. Since $X^+\subset E^+$ and $X^-\subset E^-$, we have
\[
\phi \in X^+ \cap X^- \implies \phi \in E^+ \cap E^- = \{0\}
\]
by \eqref{eq:Hdecomp}, so the sum is direct. Orthogonality follows from~\cref{prop:spectral_split}.
\end{proof}

% =====================================================================
\section{A concrete partial-wave compactness space}\label{sec:partialwaves}
% =====================================================================

In this section, we construct a concrete model in which the working-space assumptions~\textnormal{(X1)}--\textnormal{(X2)}, as well as the gradient-invariance condition~\textnormal{(X3)} introduced in~\cref{sec:assumptions}, can be checked. The construction is a standard partial-wave reduction for the three-dimensional Dirac operator. We recall the setting in some detail, because this is the point where the abstract compactness space becomes concrete.

A four-component Dirac spinor in dimension three is written as an upper two-spinor and a lower two-spinor. In spherical variables $x=r\omega$, the angular part of the Dirac operator is diagonalized by the spinor spherical harmonics $\Omega_{\kappa\mu}$. The labels $(\kappa,\mu)$ play the role of fixed angular-momentum quantum numbers. For radial scalar potentials, angular momentum is not mixed: each partial-wave channel is invariant under the Dirac operator. Thus, after fixing one channel, a spinor is described by only two radial scalar profiles, one profile for the upper component and one for the lower component.

We specialize to the lowest channel. Let $N=3$, $d=4$, and use the Dirac matrices
\[
  \alpha^j = \begin{pmatrix}0&\pauli_j\\ \pauli_j&0\end{pmatrix}, \qquad \beta = \begin{pmatrix}I_2&0\\0&-I_2\end{pmatrix},
\]
where $\pauli_1,\pauli_2,\pauli_3$ are the \textit{Pauli matrices}.
The symbols $\Omega_{\kappa\mu}$ denote the \textit{spinor spherical harmonics}, which are two-component angular spinors, that is, maps
\[
  \Omega_{\kappa\mu}:\Sphere^2\longrightarrow \CC^2,
\]
playing for the Dirac operator the same role that the scalar spherical harmonics $Y_{\ell m}$ play for the Laplacian. The index $\kappa\in\ZZ\setminus\{0\}$ labels the Dirac angular channel, while $\mu$ is the magnetic quantum number. The family $\{\Omega_{\kappa\mu}\}$ is an orthonormal basis of $L^2(\Sphere^2;\CC^2)$ adapted to total angular momentum. For the present paper we only need the lowest angular-momentum pair. Specifically, writing
\[
  \omega=(\sin\theta\cos\varphi,\sin\theta\sin\varphi,\cos\theta),
\]
we choose the phase convention
\begin{equation}\label{eq:lowest_spinor_spherical_harmonics}
  \Omega_{-1,1/2}(\omega)
  =\tfrac{1}{\sqrt{4\pi}}\binom{1}{0},
  \qquad
  \Omega_{1,1/2}(\omega)
  =-\tfrac{1}{\sqrt{4\pi}}
  \binom{\cos\theta}{e^{\imath\varphi}\sin\theta}.
\end{equation}
The first one is used for the upper two-spinor and the second one for the lower two-spinor. Other phase conventions lead to the same partial-wave space. With the convention above, the two harmonics are coupled by the radial Dirac operator through the identities
\[
  (\pauli\cdot\omega) \, \Omega_{-1,1/2}(\omega) =-\Omega_{1,1/2}(\omega), \qquad
  (\pauli\cdot\omega) \, \Omega_{1,1/2}(\omega) =-\Omega_{-1,1/2}(\omega).
\]
The special feature of this lowest channel is that both angular profiles have constant pointwise density:
\begin{equation}\label{eq:lowest_channel_constant_density}
  |\Omega_{-1,1/2}(\omega)|^2 = |\Omega_{1,1/2}(\omega)|^2 = \tfrac{1}{4\pi}.
\end{equation}
Consequently, if $\psi$ belongs to the channel defined below, then
\[
  |\psi(r\omega)|^2=\frac{|u(r)|^2+|v(r)|^2}{4\pi},
\]
so the pointwise density depends only on $r$. This simple observation is crucial later: radial scalar nonlinearities such as $b(r)|\psi|^{p-2}\psi$ preserve the same channel. The factor $\imath$ in the lower component is only a conventional choice; it is convenient because it gives the associated radial Dirac system the usual real form when the radial profiles are real.

\begin{definition}[The lowest partial-wave subspace]
Let $X_{\rm pw}$ denote the closure in $H^{1/2}(\RR^3;\CC^4)$ of the set of all smooth compactly supported spinors of the form
\begin{equation}\label{eq:Xpw_definition}
  \psi(r\omega)
  =
  \begin{pmatrix}
    u(r)\Omega_{-1,1/2}(\omega)\\
    \imath v(r)\Omega_{1,1/2}(\omega)
  \end{pmatrix},
  \qquad
  u,v\in C_c^\infty((0,\infty);\CC),
\end{equation}
where $r=|x|$ and $\omega=x/|x|$.
\end{definition}

At the $L^2$ level, the same angular ansatz with
$u,v\in L^2((0,\infty),r^2\,\mathrm dr;\CC)$ defines a closed channel, denoted
\[
  \mathcal H_{\rm pw}\coloneqq\mathcal H_{-1,1/2}.
\]
The proposition below proves, in particular, that
\[
  X_{\rm pw}=H^{1/2}(\RR^3;\CC^4)\cap\mathcal H_{\rm pw},
\]
so $X_{\rm pw}$ is simply the finite-energy part of this fixed $L^2$ partial-wave channel.

\begin{remark}
For our main result, the working space is exactly
\[
  X=X_{\rm pw}.
\]
Thus $X_{\rm pw}$ should be read as a symmetry-reduced finite-energy Dirac space: the angular dependence is fixed once and for all, while the two radial profiles remain free. The space is still infinite-dimensional, but it no longer contains arbitrary spatial translates of a fixed bump. This is why the compact Sobolev embedding, false on the full space $H^{1/2}(\RR^3;\CC^4)$, becomes true on $X_{\rm pw}$ for subcritical exponents. Since radial scalar Dirac operators reduce this channel, the spectral projections used in the strongly indefinite splitting preserve it. Later, the channel-invariance condition \textnormal{(X3)} ensures that critical points found in this reduced space are not merely constrained solutions but full weak solutions of~\eqref{eq:main}.
\end{remark}

\begin{proposition} \label{prop:Xpw_X1_X2}
Assume that the potential $V$ satisfies
\[
  V(x)=V_{\rm rad}(|x|), \qquad
  V_{\rm rad} \in L^\infty \bigl((0,\infty);\RR\bigr), \qquad
  \lim_{R\to\infty}\|V_{\rm rad}\|_{L^\infty((R,\infty))}=0,
\]
and assume that $0$ lies in a spectral gap of the perturbed operator, namely
\[
\Dmt=-\imath\alpha\cdot\nabla+m\beta+V_{\rm rad}(|x|)I_4,
\qquad
\lambda_1 := \operatorname{dist}(0,\sigma(\Dmt))>0.
\]
Then $X=X_{\rm pw}$ is a closed infinite-dimensional subspace of $\HH$ satisfying the properties~\textnormal{(X1)}--\textnormal{(X2)}. Hence, for every $q\in(2,3)$, the following embedding is compact:
\[
  X_{\rm pw}\hookrightarrow L^q(\RR^3;\CC^4).
\]
\end{proposition}

\begin{proof}
More generally, let $\mathcal H_{\kappa,\mu}$ denote the closed $L^2$-subspace with angular profile
\[
  \begin{pmatrix} u(r)\Omega_{\kappa,\mu}(\omega)\\
  \imath v(r)\Omega_{-\kappa,\mu}(\omega)\end{pmatrix},
  \qquad u,v\in L^2 \bigl((0,\infty),r^2\,\mathrm dr;\CC\bigr).
\]
The spinor spherical harmonics form an orthonormal basis of $L^2(\Sphere^2;\CC^2)$; hence
\[
  L^2(\RR^3;\CC^4) = \bigoplus_{(\kappa,\mu)}\mathcal H_{\kappa,\mu}.
\]
The corresponding angular projections commute with the Bessel-potential operator $(1-\Delta)^{s/2}$ for every $s\in\RR$; therefore, they are bounded on $H^s(\RR^3;\CC^4)$, and in particular
\[
  X_{\rm pw}=\mathcal H_{-1,1/2}\cap H^{1/2}(\RR^3;\CC^4).
\]
This identity also shows that $X_{\rm pw}$ is closed in $\HH$. We now justify
the density of the compactly supported smooth coefficient pairs in this
intersection. Let $\PiPW$ be the angular projection onto
$\mathcal H_{-1,1/2}$. Given
$\psi\in H^{1/2}(\RR^3;\CC^4)\cap\mathcal H_{-1,1/2}$, choose
$\phi_j\in C_c^\infty(\RR^3;\CC^4)$ such that
$\phi_j\to\psi$ in $H^{1/2}$. The boundedness of $\PiPW$ on $H^{1/2}$ gives
\[
  \PiPW\phi_j\longrightarrow\PiPW\psi=\psi
  \qquad\text{in }H^{1/2}.
\]
Each $\PiPW\phi_j$ has the prescribed angular profile, with smooth radial
coefficients and compact radial support, although its support may meet the
origin. Choose $\eta\in C^\infty([0,\infty);[0,1])$ such that $\eta=0$ on
$[0,1]$ and $\eta=1$ on $[2,\infty)$, and set
$\chi_\varepsilon(x)=\eta(|x|/\varepsilon)$. For fixed $j$, writing
$g_j=\PiPW\phi_j$, smoothness and compact support give, for
$0<\varepsilon\le1$,
\[
 \|(1-\chi_\varepsilon)g_j\|_{L^2}\le C_j\varepsilon^{3/2},
 \qquad
 \|(1-\chi_\varepsilon)g_j\|_{H^1}\le C_j\varepsilon^{1/2}.
\]
Interpolation between $L^2$ and $H^1$ therefore yields
\[
 \|(1-\chi_\varepsilon)g_j\|_{H^{1/2}}
 \le C_j\varepsilon.
\]
Consequently,
\[
  \chi_\varepsilon\PiPW\phi_j\longrightarrow\PiPW\phi_j
  \qquad\text{in }H^{1/2}
\]
as $\varepsilon\downarrow0$. Multiplication by a radial scalar function
preserves the channel. A diagonal choice of $j$ and $\varepsilon$ therefore
produces spinors of the form~\eqref{eq:Xpw_definition} converging to $\psi$.
Finally, the space is infinite-dimensional because the coefficient pairs
$(u,v)\in C_c^\infty((0,\infty);\CC)^2$ produce an infinite-dimensional
family of linearly independent spinors.

For a radial scalar potential, the Dirac operator does not mix angular channels. More precisely, the standard partial-wave decomposition of the three-dimensional Dirac operator~\cite[Section~4.6]{Thaller1992} gives
\[
  \Dmt=\bigoplus_{(\kappa,\mu)}\widetilde{\mathcal D}_{m,\kappa,\mu},
  \qquad
  \dom\bigl(\widetilde{\mathcal D}_{m,\kappa,\mu}\bigr)=\dom \bigl(\Dmt\bigr)\cap\mathcal H_{\kappa,\mu},
\]
where each $\widetilde{\mathcal D}_{m,\kappa,\mu}$ is the associated radial Dirac system on the coefficient pair $(u,v)$. Thus every $\mathcal H_{\kappa,\mu}$ is a reducing subspace for $\Dmt$. By the spectral theorem, every bounded Borel function of $\Dmt$ preserves each reducing subspace. Applying this to the spectral projections $P_m=\chi_{(0,\infty)}(\Dmt)$ and $Q_m=\chi_{(-\infty,0)}(\Dmt)$, and using their boundedness on $H^{1/2}$ from~\cref{prop:spectral_split}, gives the desired inclusions:
\[
  P_mX_{\rm pw}\subset X_{\rm pw},\qquad Q_mX_{\rm pw}\subset X_{\rm pw}.
\]
It remains to check that the two projected spaces are infinite-dimensional. On the reduced channel $\mathcal H_{-1,1/2}$, the perturbation $V_{\rm rad}(r)I_4$ is relatively compact with respect to the free radial Dirac operator. Indeed, a sequence bounded in the free radial Dirac graph norm is bounded locally in $H^1$; on a fixed ball, multiplication by $V_{\rm rad}$ is compact from $H^1$ to $L^2$ by Rellich's theorem, while on the complement of a large ball the $L^\infty$-norm of $V_{\rm rad}$ is arbitrarily small by the assumed essential tail decay. By Weyl's essential spectrum theorem~\cite[Chapter IV, Section 5]{Kato1995},
\[
  \sigma_{\mathrm{ess}} \bigl(\Dmt|_{\mathcal H_{-1,1/2}} \bigr)
  =(-\infty,-m]\cup[m,+\infty).
\]
Both positive and negative essential spectral branches are non-empty and unbounded. Since $0$ lies in a spectral gap by assumption, the projections of $\Dmt|_{\mathcal H_{-1,1/2}}$ onto $(0,\infty)$ and $(-\infty,0)$ have infinite rank. Moreover, this infinite-dimensionality is already visible on the form domain: choosing vectors with spectral support in bounded intervals contained in $(m,\infty)$ and in $(-\infty,-m)$ gives infinitely many linearly independent vectors in
$\dom(|\Dmt|^{1/2})\cap\mathcal H_{-1,1/2}=X_{\rm pw}$. This proves~\textnormal{(X1)}.

Finally, the fixed angular profile gives a finite-angular-mode version of the radial compactness theorem. Indeed, using~\eqref{eq:lowest_channel_constant_density}, we have
\[
  |\psi(r\omega)|^2=\frac{|u(r)|^2+|v(r)|^2}{4\pi},
\]
and the $L^q$-norm is equivalent to the corresponding profile norm. The
$H^{1/2}$-equivalence between the spinor norm and the product
Bessel-potential norm of the two radial profiles is proved in
\cref{lem:finite_angular_width_transfer}. The compact embedding theorem for radial and finite-angular-mode Sobolev spaces (see, for instance,~\cite{SickelSkrzypczakVybiral2012}), applied componentwise to the two radial coefficients, yields compactness of
\[
  X_{\rm pw}\hookrightarrow L^q(\RR^3;\CC^4),\qquad 2<q<3.
\]
\end{proof}

\begin{lemma}\label{lem:finite_angular_width_transfer}
Identify a scalar radial function on $\RR^3$ with its radial profile. Let
\[
  \mathcal R^{1/2}_2=RH^{1/2}_2(\RR^3)\oplus RH^{1/2}_2(\RR^3)
\]
and define
\[
  \mathscr U(u,v)(r\omega)=
  \begin{pmatrix}
    u(r) \, \Omega_{-1,1/2}(\omega)\\
    \imath v(r) \, \Omega_{1,1/2}(\omega)
  \end{pmatrix}.
\]
Then $\mathscr U$ is an isomorphism from $\mathcal R^{1/2}_2$, endowed
with its product Bessel-potential norm, onto $X_{\rm pw}$, endowed with the
$H^{1/2}$-norm. Moreover, for every $2<q<3$, one has the norm equivalence
\[
  \|\mathscr U(u,v)\|_{L^q(\RR^3;\CC^4)}
  \asymp
  \left(\|u\|_{L^q((0,\infty),r^2\,\mathrm dr)}^q
       +\|v\|_{L^q((0,\infty),r^2\,\mathrm dr)}^q\right)^{1/q}.
\]
Consequently, for $2<p<3$, if the scalar radial embedding
\[
  T:RH^{1/2}_2(\RR^3)\hookrightarrow RL^p(\RR^3)
\]
has approximation numbers bounded by $Cj^{-\alpha}$, then
\[
  X_{\rm pw}\hookrightarrow L^p(\RR^3;\CC^4)
\]
has approximation numbers bounded by $C'j^{-\alpha}$. The same bound
holds after restricting the domain to a closed Hilbert subspace of
$X_{\rm pw}$ (e.g., to $X^+$ or $X_B^+$) whose norm is equivalent to the $H^{1/2}$-norm.
\end{lemma}

\begin{proof}
For $\kappa\in\{-1,1\}$ write
\[
  \mathscr U_\kappa f(r\omega)=f(r) \, \Omega_{\kappa,1/2}(\omega),
\]
and let $Y_\kappa^s\subset H^s(\RR^3;\CC^2)$ be the closed subspace with
this fixed angular profile.  The corresponding orbital angular momenta are
$\ell_{-1}=0$ and $\ell_1=1$.

\smallskip
\noindent\emph{Step 1: the $L^2$ endpoint.}
Since $\Omega_{\kappa,1/2}$ is normalized in $L^2(\Sphere^2;\CC^2)$,
Fubini's theorem gives, first for smooth compactly supported profiles,
\begin{equation}\label{eq:angular_L2_endpoint}
  \|\mathscr U_\kappa f\|_{L^2(\RR^3)}^2
  =\int_0^\infty |f(r)|^2r^2\,\mathrm dr.
\end{equation}
Thus $\mathscr U_\kappa$ and the coefficient-extraction map
\[
  \mathscr V_\kappa F(r)
  =\int_{\Sphere^2}\langle F(r\omega),
      \Omega_{\kappa,1/2}(\omega)\rangle_{\CC^2}\,\mathrm d\omega
\]
are mutually inverse bounded maps between the radial $L^2$ coefficient space
and $Y_\kappa^0$.

\smallskip
\noindent\emph{Step 2: the $H^1$ endpoint.}
The spinor spherical harmonic $\Omega_{\kappa,1/2}$ satisfies
\[
 -\Delta_{\Sphere^2}\Omega_{\kappa,1/2}
 =\ell_\kappa(\ell_\kappa+1)\Omega_{\kappa,1/2}.
\]
The polar-coordinate identity
$|\nabla F|^2=|\partial_rF|^2+r^{-2}|\nabla_{\Sphere^2}F|^2$ therefore yields
\begin{equation}\label{eq:angular_H1_endpoint}
  \|\nabla(\mathscr U_\kappa f)\|_{L^2(\RR^3)}^2
  =\int_0^\infty |f'(r)|^2r^2\,\mathrm dr
   +\ell_\kappa(\ell_\kappa+1)
      \int_0^\infty |f(r)|^2\,\mathrm dr.
\end{equation}
The three-dimensional Hardy inequality, applied to the radial scalar
function $x\mapsto f(|x|)$, is
\begin{equation}\label{eq:radial_Hardy_profile}
  \int_0^\infty |f(r)|^2\,\mathrm dr
  \le 4\int_0^\infty |f'(r)|^2r^2\,\mathrm dr.
\end{equation}
Combining \eqref{eq:angular_L2_endpoint}--\eqref{eq:radial_Hardy_profile}
shows that the $H^1$-norm of $\mathscr U_\kappa f$ is equivalent to the
radial scalar $H^1$-norm of $f$. The coefficient-extraction map is bounded
in the reverse direction by the same identities. By density, these maps
extend to mutually inverse isomorphisms between the radial $H^1$ coefficient
space and $Y_\kappa^1$.

\smallskip
\noindent\emph{Step 3: interpolation.}
The angular projection
\[
  P_\kappa F(r\omega)
  =\Omega_{\kappa,1/2}(\omega)
    \int_{\Sphere^2}\langle F(r\eta),
      \Omega_{\kappa,1/2}(\eta)\rangle_{\CC^2}\,\mathrm d\eta
\]
is bounded on both $L^2$ and $H^1$. Indeed, $P_\kappa$ is an orthogonal
projection on each sphere, commutes with the radial derivative, and the
spherical-harmonic decomposition gives, for almost every $r>0$,
\[
 \|\partial_rP_\kappa F(r,\cdot)\|_{L^2(\Sphere^2)}
 \le \|\partial_rF(r,\cdot)\|_{L^2(\Sphere^2)},
 \qquad
 \|\nabla_{\Sphere^2}P_\kappa F(r,\cdot)\|_{L^2(\Sphere^2)}
 \le \|\nabla_{\Sphere^2}F(r,\cdot)\|_{L^2(\Sphere^2)}.
\]
After integration in the radial variable, the polar-coordinate decomposition
used in \eqref{eq:angular_H1_endpoint} proves the $H^1$ bound. Hence
$Y_\kappa^0$ and $Y_\kappa^1$ form a
compatible complemented couple. Radial averaging is likewise a bounded
projection on both $L^2$ and $H^1$, so the radial source spaces form a
compatible complemented couple as well. Complex interpolation gives
\[
 [Y_\kappa^0,Y_\kappa^1]_{1/2}
 =Y_\kappa^{1/2},
 \qquad
 [RL^2,RH^1_2]_{1/2}=RH^{1/2}_2.
\]
Interpolating the mutually inverse endpoint maps proves that
$\mathscr U_\kappa:RH^{1/2}_2\to Y_\kappa^{1/2}$ is an isomorphism.  Taking
the direct sum of the $\kappa=-1$ and $\kappa=1$ components proves the $H^{1/2}$-norm equivalence for $\mathscr U$ and identifies its range
with $X_{\rm pw}$.

\smallskip
\noindent\emph{Step 4: target norms and approximation numbers.}
In the lowest channel, \eqref{eq:lowest_channel_constant_density} gives
\[
 |\mathscr U(u,v)(r\omega)|^2
 =\frac{|u(r)|^2+|v(r)|^2}{4\pi}.
\]
After integration over $\Sphere^2$, the equivalence of the Euclidean
$\ell^2$- and $\ell^q$-norms on $\CC^2$ yields the displayed $L^q$-norm
equivalence.

Under the coefficient isomorphisms, the partial-wave embedding is the
composition of bounded isomorphisms with $T\oplus T$. Equip the source
direct sum with its $\ell^2$ product norm and the target direct sum with its
$\ell^p$ product norm; replacing these by equivalent product norms only
changes the constants below. If $A_j$ has rank less than $j$ and
$\|T-A_j\|\le 2a_j(T)$, then $A_j\oplus A_j$ has rank less than $2j$.
Consequently, for a constant $C_\oplus>0$ depending only on the chosen
product norms,
\[
  a_{2j}(T\oplus T)\le C_\oplus a_j(T)
  \le C_\oplus Cj^{-\alpha}.
\]
Monotonicity of approximation numbers gives
$a_n(T\oplus T)\le C_1n^{-\alpha}$ for every $n$. The ideal property then
transfers this estimate through the coefficient isomorphisms. Finally,
restriction to a closed subspace cannot increase approximation numbers, and
equivalent Hilbert norms change only the multiplicative constant. These power
bounds are unchanged when the complex spaces are regarded as real spaces: if
$T_{\RR}$ denotes the realification of a complex-linear operator $T$, then
\[
 a^{\RR}_{2j}(T_{\RR})\le a^{\CC}_j(T),
\]
because a complex operator of rank less than $j$ has real rank less than
$2j$. Monotonicity therefore gives the same decay exponent after realification.
\end{proof}

\begin{remark}
The value of the exponent used later is supplied by the radial approximation
number theorem of \citet{DotaSkrzypczak2025};
\cref{lem:finite_angular_width_transfer} only transfers that scalar radial
estimate to the two-component partial-wave channel.
\end{remark}

\begin{lemma}\label{lem:nested_tail_basis}
Let $T:H\to B$ be a compact linear embedding from an infinite-dimensional separable Hilbert space into a Banach space. Assume that its approximation numbers satisfy
\[
  a_j(T)\le Cj^{-\alpha}
\]
for some $\alpha>0$. Then there is an orthonormal basis $(e_j)_{j\ge1}$ of $H$ such that, with
\[
  Z_k=\overline{\operatorname{span}}\bigl\{e_j \ : \ j\ge k+1 \bigr\},
\]
one has, for every $k \ge 1$,
\[
  \sup_{\substack{z\in Z_k\\ \|z\|_H=1}}\|Tz\|_B\le C'k^{-\alpha}.
\]
\end{lemma}

\begin{proof}
We use the relation between approximation and Gelfand numbers,
\[
  c_m(T)\coloneqq \inf_{\operatorname{codim}M<m}\|T|_M\|\le a_m(T).
\]
For each dyadic $m_\ell=2^\ell$ choose a closed subspace $L_\ell\subset H$ with $\operatorname{codim}L_\ell<m_\ell$ and
\[
  \|T|_{L_\ell}\|\le 2C m_\ell^{-\alpha}.
\]
Set $M_0=H$ and, for $\ell\ge1$, $M_\ell=\cap_{j=1}^\ell L_j$. Then $(M_\ell)$ is a decreasing sequence of closed subspaces,
\[
  r_\ell\coloneqq\operatorname{codim}M_\ell\le \sum_{j=1}^\ell 2^j<2^{\ell+1},
  \qquad
  \|T|_{M_\ell}\|\le 2C2^{-\alpha\ell}.
\]
Moreover, $\bigcap_{\ell\ge1}M_\ell=\{0\}$. Indeed, if $x$ belongs to the intersection, then
$\|Tx\|_B\le 2C2^{-\alpha\ell}\|x\|_H$ for every $\ell$, hence $Tx=0$, and the embedding property gives $x=0$.

We now construct the basis. Since $M_\ell\subset M_{\ell-1}$ and both subspaces have finite codimension, the orthogonal difference
\[
  W_\ell=M_{\ell-1}\cap M_\ell^\perp
\]
is finite-dimensional, with dimension $r_\ell-r_{\ell-1}$. Choose an orthonormal basis in each $W_\ell$ and list these bases block by block. The closed span of the first $r_\ell$ vectors is $M_\ell^\perp$, and therefore the closed tail after the first $r_\ell$ vectors is $M_\ell$. Since the intersection of the $M_\ell$ is zero, the union of the blocks is a complete orthonormal basis of $H$.

If $2^{\ell+1}\le k<2^{\ell+2}$, then $r_\ell<2^{\ell+1}\le k$, so the $k$-tail is contained in $M_\ell$. Hence, for every vector $z$ in that tail,
\[
  \|Tz\|_B\le 2C2^{-\alpha\ell}\|z\|_H\le C'k^{-\alpha}\|z\|_H.
\]
This proves the claim for $k\ge4$ after changing constants; enlarging the
constant once more covers $k \in \{1,2,3\}$.
\end{proof}

\begin{remark} \label{remark:spectral_gap}
If $\|V\|_{L^\infty}<m$, then $0\notin\sigma(\Dmt)$. Indeed, $\Dm$ is invertible and
\[
\|V\Dm^{-1}\|_{L^2\to L^2}\le \|V\|_{L^\infty}/m<1.
\]
Hence $I+V\Dm^{-1}$ is also invertible on $L^2$, and
\[
  \Dmt=(I+V\Dm^{-1})\Dm,
  \qquad
  \Dmt^{-1}=\Dm^{-1}(I+V\Dm^{-1})^{-1}.
\]
Thus $0\notin\sigma(\Dmt)$. Since the resolvent set is open, this implies
$\operatorname{dist}(0,\sigma(\Dmt))>0$, giving a class of radial potentials
satisfying the gap condition.
\end{remark}

\begin{proposition}[Channel criterion] \label{prop:Xpw_X3_criterion}
Let $V$ be as in~\cref{prop:Xpw_X1_X2}. Let
$\PiPW$ denote the $L^2$-orthogonal projection onto the closed partial-wave
channel $\mathcal H_{-1,1/2}$. Its restriction to
$H^{1/2}(\RR^3;\CC^4)$ has range $X_{\rm pw}$, and it extends continuously
to $H^s(\RR^3;\CC^4)$ for $-1/2\le s\le 1/2$ by duality and interpolation.
Suppose that for every $\psi\in X_{\rm pw}$ the nonlinear term
\[
  h_\psi(x) \coloneqq f_0(x,\psi(x))+g(x,\psi(x))
\]
defines an element of $H^{-1/2}(\RR^3;\CC^4)$ and satisfies
\[
  \PiPW h_\psi=h_\psi \qquad\text{in }H^{-1/2}(\RR^3;\CC^4).
\]
Equivalently, in distributional angular variables,
\[
  h_\psi(r\omega) =
  \begin{pmatrix}
    A_\psi(r)\Omega_{-1,1/2}(\omega)\\[1mm]
    \imath B_\psi(r)\Omega_{1,1/2}(\omega)
  \end{pmatrix}
\]
for suitable radial distributions $A_\psi,B_\psi$. Then assumption~\textnormal{(X3)} holds:
\[
  \nabla_\HH E(X_{\rm pw})\subset X_{\rm pw}.
\]
\end{proposition}

\begin{proof}
The growth assumptions imply $h_\psi\in L^2+L^{p'}\hookrightarrow H^{-1/2}$
for $\psi\in\HH$; in the lower-order variants this follows instead from the
weighted estimate
\[
  \|b_\tau |\psi|^{\tau-1}\|_{L^{p'}}
  \le \|b_\tau\|_{L^{p/(p-\tau)}}\|\psi\|_{L^p}^{\tau-1},
  \qquad 1<\tau<p.
\]
The nonlinear contribution to the $\HH$-gradient is, up to the irrelevant
minus sign, the Riesz representative $|\Dmt|^{-1}h_\psi\in H^{1/2}$. Since
the angular projection $\PiPW$ reduces $\Dmt$ on $L^2$, it commutes with
every bounded Borel function of $\Dmt$ on $L^2$. In particular it commutes
with $|\Dmt|^{-1}$ on $L^2$. Because $\PiPW$ is bounded on $H^{-1/2}$ and
$H^{1/2}$, and because both sides are continuous maps
$H^{-1/2}\to H^{1/2}$, this commutation identity extends by density to
\[
  \PiPW |\Dmt|^{-1}h_\psi
  =|\Dmt|^{-1}\PiPW h_\psi
  =|\Dmt|^{-1}h_\psi.
\]
Thus the nonlinear gradient contribution belongs to $X_{\rm pw}$. The linear
part $\psi^+-\psi^-$ belongs to $X_{\rm pw}$ by~\cref{prop:Xpw_X1_X2}. Thus
the full gradient belongs to $X_{\rm pw}$.
\end{proof}

\begin{proposition}\label{prop:restricted_full_solution}
Suppose that the linear operator and the nonlinear Nemytskii term preserve the lowest partial-wave channel in the sense of~\cref{prop:Xpw_X1_X2,prop:Xpw_X3_criterion}. If $\psi\in X_{\rm pw}$ is a critical point of the restricted functional $E|_{X_{\rm pw}}$, then $\psi$ is a critical point of $E$ on the full space $H^{1/2}(\RR^3;\CC^4)$.  
\end{proposition}

\begin{proof}
Let $\nabla_{\HH}E(\psi)$ be the full Hilbert gradient. Criticality on $X_{\rm pw}$ says that this vector is orthogonal to $X_{\rm pw}$. By the channel invariance criterion \textnormal{(X3)}, $\nabla_{\HH}E(\psi)\in X_{\rm pw}$. Hence $\nabla_{\HH}E(\psi)=0$. This is exactly the weak Euler--Lagrange equation on $H^{1/2}(\RR^3;\CC^4)$, concluding the proof.
\end{proof}

\begin{remark}
Equivalently, such a function $\psi \in X_{\rm pw}$ satisfies the Dirac equation weakly against all test spinors, not only against partial-wave variations.
\end{remark}

\begin{remark} \label{rem:exact_power_channel}
For the model $F_0(x,\psi)=b(|x|)|\psi|^p/p$, if
$\psi\in X_{\rm pw}$ has coefficients $(u,v)$ as in
\eqref{eq:Xpw_definition}, then
\[
  f_0(x,\psi)=
  b(r)\left(\frac{|u(r)|^2+|v(r)|^2}{4\pi}\right)^{(p-2)/2}\psi,
  \qquad r=|x|.
\]
Thus the exact-power core preserves the lowest partial-wave channel.
\end{remark}

\begin{example}
\label{ex:Xpw_model_perturbation}
Let $p\in(2,3)$, fix $1<\tau<p/2$, and let $b\in C_b([0,\infty))$ satisfy
$0<b_0\le b(r)\le b_1$. Set
\[
  F_0(x,\psi)=\frac{b(|x|)}p|\psi|^p.
\]
Let $a\in C_c^\infty((0,\infty))$, $a\ge0$, $a\not\equiv0$, and extend it radially. Define
\[
  \eta(\omega)=
  \begin{pmatrix}\Omega_{-1,1/2}(\omega)\\0\end{pmatrix},
  \qquad
  \ell_\eta(\omega,\psi)=\RE\inner{\eta(\omega)}{\psi}_{\CC^4},
\]
and, for a fixed $\zeta>0$ with $\zeta\ne1$, define
\[
  s_+=\max\{s,0\},\qquad s_-=\max\{-s,0\},
\]
and
\[
  K_{\tau,\zeta}(s)=\frac1\tau\bigl((s_+)^\tau+\zeta(s_-)^\tau\bigr).
\]
The function $K_{\tau,\zeta}$ is $C^1$, convex, and not even. For an arbitrary
amplitude $\varepsilon>0$ put
\begin{equation}\label{eq:Xpw_G_model}
  G(x,\psi)=
  \begin{cases}
   \varepsilon a(|x|)
   K_{\tau,\zeta}\!\left(\ell_\eta(x/|x|,\psi)\right),&x\ne0,\\
   0,&x=0.
  \end{cases}
\end{equation}
Because $a$ vanishes on a neighborhood of the origin, this extension is
continuous there and the angular factor is harmless. Moreover,
\[
 |K'_{\tau,\zeta}(s)|\le C|s|^{\tau-1},
 \qquad K'_{\tau,\zeta}(s)s=\tau K_{\tau,\zeta}(s).
\]
Since $a$ is bounded and compactly supported, these identities verify
\textnormal{(G1)}--\textnormal{(G4)} with weights proportional to $a$.
Convexity gives \textnormal{(G5)} with $\mu\equiv0$, independently of
$\varepsilon$.
For $\psi\in X_{\rm pw}$,
\[
 \ell_\eta(\omega,\psi(r\omega))=(4\pi)^{-1}\RE u(r),
\]
so $g=\nabla_\psi^{\RR}G$ has the same channel form as
\eqref{eq:Xpw_definition}, with vanishing lower radial coefficient.
Thus \cref{prop:Xpw_X3_criterion} gives the required channel invariance.
\end{example}

\begin{remark}[Concrete homogeneous partial-wave model] \label{cor:concrete_partial_wave_model}
Let $N=3$, let $V(x)=V_r(|x|)$ with
\[
V_r\in L^\infty(0,\infty;\RR), \qquad
\lim_{R\to\infty}\|V_r\|_{L^\infty((R,\infty))}=0,
\qquad \|V\|_{L^\infty}<m,
\]
and fix $2<p<3$. Let $b\in C_b([0,\infty))$ satisfy $0<b_0\le b(r)\le b_1<\infty$, and set
\[
  F_0(x,\psi)=\frac{b(|x|)}p|\psi|^p.
\]
The gap condition on $V$ gives the spectral gap by~\cref{remark:spectral_gap}, and~\cref{prop:Xpw_X1_X2} verifies \textnormal{(X1)}--\textnormal{(X2)} on $X_{\rm pw}$. The homogeneous nonlinearity above satisfies \textnormal{(F0)}--\textnormal{(F4)} and preserves the partial-wave channel. This is the concrete model used below. \cref{ex:Xpw_model_perturbation} gives a non-even convex lower-order perturbation satisfying the analytic, semiconvexity, and channel-invariance assumptions of the perturbation-from-symmetry theorem.
\end{remark}

% =====================================================================

\section{Assumptions and the energy functional}\label{sec:assumptions}
% =====================================================================

\subsection{Assumptions on the potential and on the working space}

We keep the assumptions on the scalar potential $V : \RR^N \to \RR$ introduced in~\cref{sec:setting}. Specifically:

\begin{enumerate}
\item[\textbf{(V1)}]
  The potential $V$ is real-valued and belongs to $L^\infty(\RR^N)$.

\item[\textbf{(V2)}]
  The form domain of $|\Dmt|^{1/2}$ is $H^{1/2}(\RR^N;\CC^d)$, and its graph norm is equivalent to the standard $H^{1/2}$-norm.

\item[\textbf{(V3)}]
  The origin lies in a gap of the spectrum of the operator
  \[
    \Dmt = \Dm + V(x)I_d.
  \]
  Specifically, we set
  \[
    \lambda_1 \coloneqq \operatorname{dist}\bigl(0,\sigma(\Dmt)\bigr) > 0.
  \]
\end{enumerate}

In addition, we fix a closed subspace $X \subset \HH$ satisfying \textnormal{(X1)}--\textnormal{(X2)} from~\cref{ssec:compactness}. After the energy functional $E$ is defined in~\eqref{eq:energy}, we also require the following compatibility condition:

\begin{enumerate}
\item[\textbf{(X3)}]
  The $\HH$-gradient of the full energy functional $E:\HH\to\RR$ defined in~\eqref{eq:energy} leaves $X$ invariant, namely
  \[
    \nabla_\HH E(X) \subset X.
  \]
\end{enumerate}

We emphasize that \textnormal{(X3)} is an \emph{additional assumption} on the chosen compactness subspace $X$. Finally, for every $q \in (2,2^*_{1/2})$, we denote by $S_q>0$ the continuity constant of the embedding $X \hookrightarrow L^q(\RR^N;\CC^d)$.

\subsection{Assumptions on the symmetric nonlinearity}

Throughout this section we identify $\CC^d$ with $\RR^{2d}$. If $H(x,\cdot) : \CC^d\to\RR$ is of class $C^1$ in the real sense, we denote by $\nabla_\psi^{\RR}H(x,\psi)\in\CC^d$ its real gradient, i.e.
\[
  \mathrm{d}_\psi H(x,\psi)[\varphi] = \RE\inner{\nabla_\psi^{\RR}H(x,\psi)}{\varphi}_{\CC^d} \qquad \text{for all } \varphi \in \CC^d.
\]
We assume that there exists a function $F_0 : \RR^N \times \CC^d \to \RR$ such that:

\begin{enumerate}
\item[\textbf{(F0)}] 
    The function is even in the second variable, i.e.
  \[
    F_0(x,-\psi) = F_0(x,\psi) \qquad \text{for all } (x,\psi) \in \RR^N \times \CC^d.
  \]

\item[\textbf{(F1)}]
  The function $F_0$ is continuous on $\RR^N \times \CC^d$, and for every
  $x \in \RR^N$ the map $F_0(x,\cdot)$ is of class $C^1$ on $\CC^d$ in
  the real sense. Its real gradient
  \[
    f_0(x,\psi) \coloneqq \nabla_\psi^{\RR}F_0(x,\psi)
  \]
  is a Carath\'eodory map. Moreover,
  \[
    F_0(x,0)=0 \qquad \text{for all } x\in\RR^N.
  \]
  Then $f_0(x,-\psi)=-f_0(x,\psi)$ for all $(x,\psi)$.

\item[\textbf{(F2)}]
  There exist a constant $C_0 > 0$ and an exponent
  \[
    2 < p < 2^*_{1/2} \coloneqq \frac{2N}{N-1}
  \]
  such that
  \[
    \abs{f_0(x,\psi)}
    \leq C_0\bigl(\abs{\psi} + \abs{\psi}^{p-1}\bigr)
    \qquad \text{for all } (x,\psi) \in \RR^N \times \CC^d.
  \]

\item[\textbf{(F3)}]
  For every $x \in \RR^N$ and every $\psi \neq 0$,
  \[
    0 < p\,F_0(x,\psi) \leq \RE \inner{f_0(x,\psi)}{\psi}_{\CC^d}.
  \]
  Moreover, there exists $c_0 > 0$ such that
  \[
    F_0(x,\psi) \geq c_0 \abs{\psi}^p \qquad \text{for all } (x,\psi) \in \RR^N \times \CC^d.
  \]

\item[\textbf{(F4)}]
  Uniformly in $x$,
  \[
    f_0(x,\psi) = o(\abs{\psi}) \qquad \text{as } \abs{\psi} \to 0.
  \]
\end{enumerate}

Since $F_0(x,0)=0$ and $f_0=\nabla_\psi^{\RR}F_0$, the fundamental theorem of calculus yields a primitive given by
\begin{equation}\label{eq:F0_radial}
  F_0(x,\psi) = \int_0^1 \RE \inner{f_0(x,t\psi)}{\psi}_{\CC^d}\dt.
\end{equation}
As a consequence of~\textnormal{(F2)} and \textnormal{(F4)}, for every $\varepsilon > 0$ there is $C_\varepsilon > 0$ such that
\begin{equation}\label{eq:f0_eps_growth}
  \abs{f_0(x,\psi)} \leq \varepsilon \abs{\psi} + C_\varepsilon \abs{\psi}^{p-1} \qquad \text{for all } (x,\psi) \in \RR^N \times \CC^d.
\end{equation}
Moreover, by~\eqref{eq:F0_radial} and~\textnormal{(F2)}, there exists $C > 0$ such that
\begin{equation}\label{eq:F0_growth}
  \abs{F_0(x,\psi)} \leq C\bigl(\abs{\psi}^2 + \abs{\psi}^p\bigr) \qquad \text{for all } (x,\psi) \in \RR^N \times \CC^d.
\end{equation}
Combining~\eqref{eq:F0_radial} with~\eqref{eq:f0_eps_growth}, for every $\varepsilon > 0$ there exists $C_\varepsilon > 0$ such that
\begin{equation}\label{eq:F0_eps_growth}
  \abs{F_0(x,\psi)} \leq \varepsilon \abs{\psi}^2 + C_\varepsilon \abs{\psi}^p \qquad \text{for all } (x,\psi) \in \RR^N \times \CC^d.
\end{equation}

\begin{remark}\label{rem:model}
The model potential $F_0(x,\psi) = \tfrac1p \abs{\psi}^p$ satisfies all the required properties \textnormal{(F0)}--\textnormal{(F4)} and its real gradient is given by $f_0(x,\psi) = \abs{\psi}^{p-2}\psi$.
\end{remark}

\subsection{Assumptions on the perturbation}

Assume that there exists a real-valued function $G : \RR^N \times \CC^d \to \RR$ such that:

\begin{enumerate}
\item[\textbf{(G1)}]
  The function $G$ is continuous on $\RR^N \times \CC^d$, and for every
  $x \in \RR^N$ the map $G(x,\cdot)$ is of class $C^1$ on $\CC^d$ in
  the real sense. Its real gradient
  \[
    g(x,\psi) \coloneqq \nabla_\psi^{\RR}G(x,\psi)
  \]
  is a Carath\'eodory map. Moreover,
  \[
    G(x,0)=0 \qquad \text{for all } x\in\RR^N.
  \]

\item[\textbf{(G2)}]
  There exists an exponent $\tau \in (1,2)$ and a non-negative function
  \[
    b_\tau\in L^\infty(\RR^N)\cap L^{p/(p-\tau)}(\RR^N)
  \]
  such that, for all $(x,\psi)\in\RR^N\times\CC^d$,
  \begin{equation}\label{eq:g_mixed_growth}
    |g(x,\psi)| \le b_\tau(x)|\psi|^{\tau-1}.
  \end{equation}

\item[\textbf{(G3)}]
  There exists a non-negative function
  \[
    a_\tau \in L^\infty(\RR^N)\cap L^{p/(p-\tau)}(\RR^N)
  \]
  such that, for all $(x,\psi)$,
  \begin{equation}\label{eq:G_split_bound}
    |G(x,\psi)| \le a_\tau(x)|\psi|^\tau.
  \end{equation}

\item[\textbf{(G4)}]
  There exists a non-negative function
  \[
    c_\tau \in L^\infty(\RR^N)\cap L^{p/(p-\tau)}(\RR^N)
  \]
  such that, for all $(x,\psi)$,
  \begin{equation}\label{eq:G_virial_bound}
    \left| \tfrac12\RE\inner{g(x,\psi)}{\psi}_{\CC^d}-G(x,\psi) \right| \le c_\tau(x)|\psi|^\tau.
  \end{equation}

\item[\textbf{(G5)}]
  There exists a non-negative function $\mu\in L^\infty(\RR^N)$ with
  \begin{equation}\label{eq:G_semiconvex_gap}
    \|\mu\|_{L^\infty}<\lambda_1
  \end{equation}
  such that, for all $x\in\RR^N$ and $z,w\in\CC^d$,
  \begin{equation}\label{eq:G_semiconvex}
    \RE\inner{g(x,z)-g(x,w)}{z-w}_{\CC^d}
    \ge-\mu(x)|z-w|^2.
  \end{equation}
\end{enumerate}

For later use we write
\[
  A_\tau=\|a_\tau\|_{L^{p/(p-\tau)}},\qquad
  B_\tau=\|b_\tau\|_{L^{p/(p-\tau)}},\qquad
  C_\tau=\|c_\tau\|_{L^{p/(p-\tau)}}.
\]
Condition \textnormal{(G5)} is a quantitative semiconvexity assumption. It is automatic with $\mu\equiv0$ when $G(x,\cdot)$ is convex. The main nonsymmetric theorem assumes $1<\tau<p/2$. Under \textnormal{(G1)}--\textnormal{(G2)}, condition \textnormal{(G3)} follows from the formula below with $a_\tau=b_\tau/\tau$; we keep it as a separate hypothesis for later reference. Since $G(x,0)=0$ and $g=\nabla_\psi^{\RR}G$, we also have
\begin{equation}\label{eq:G_radial}
  G(x,\psi) = \int_0^1 \RE \inner{g(x,t\psi)}{\psi}_{\CC^d}\dt.
\end{equation}

\begin{remark}\label{rem:tail_integrability}
The integrability condition on $a_\tau$, and similarly on $b_\tau$ and $c_\tau$, is precisely the weighted H\"older condition associated with the compact $L^p$ embedding:
\[
  \int_{\RR^N} a_\tau|\psi|^\tau\dx
  \le A_\tau\norm{\psi}_{L^p}^\tau.
\]
For the derivative of the perturbation, the exponents are likewise exactly
compatible:
\[
\int_{\RR^N} b_\tau|\psi|^{\tau-1}|\varphi|\,dx
\le
\|b_\tau\|_{L^{p/(p-\tau)}}
\|\psi\|_{L^p}^{\tau-1}
\|\varphi\|_{L^p},
\qquad
\tfrac{p-\tau}{p}+\tfrac{\tau-1}{p}+\tfrac1p=1.
\]
\end{remark}

\begin{example}\label{ex:model}
Take $N=3$, $m>0$, $V\equiv0$, $1<\tau<2<p<3$, and
\[
 F_0(\psi)=\tfrac1p|\psi|^p.
\]
Let $L:\CC^4\to\RR$ be a non-zero real linear functional, let
$a\in C_c^\infty(\RR^3)$ be non-negative and non-zero, and choose $\eta>0$, $\eta\ne1$.
Set $s_\pm = \max\{ \pm s, 0 \}$ and define the functions
\[
 K_{\tau,\eta}(s)=\tfrac1\tau\bigl((s_+)^\tau+\eta(s_-)^\tau\bigr),
 \qquad
 G(x,\psi)=a(x)K_{\tau,\eta}(L\psi).
\]
Since 
\[
|K'_{\tau,\eta}(s)|\le C|s|^{\tau-1}, \qquad  K'_{\tau,\eta}(s)s=\tau K_{\tau,\eta}(s),
\]
the function $G$ is $C^1$, convex in
$\psi$, not even, and satisfies
\textnormal{(G1)}--\textnormal{(G5)} with $\mu\equiv0$. Multiplying $a$ by
an arbitrary positive constant does not affect the semiconvexity condition,
so no amplitude smallness is required in this convex example. To meet the
quantitative hypothesis of the multiplicity theorem, one additionally chooses
$\tau<p/2$. This example verifies the perturbation assumptions
\textnormal{(G1)}--\textnormal{(G5)}; the separate invariance condition
\textnormal{(X3)} must still be checked for the chosen compactness subspace.
\end{example}

\subsection{The energy functional}

We define the full potential as the sum of the even part and the perturbation:
\[
  F(x,\psi) \coloneqq F_0(x,\psi) + G(x,\psi).
\]
We first record the weighted H\"older estimate used below. Whenever
$1<r<p$ and $a_r\in L^{p/(p-r)}(\RR^N)$,
\begin{equation}\label{eq:holder_weighted}
  \int_{\RR^N} a_r(x)|\psi|^r\dx
  \le \|a_r\|_{L^{p/(p-r)}}\|\psi\|_{L^p}^r.
\end{equation}
For the endpoint $r=p$ we use the same estimate with the convention
$p/(p-r)=\infty$; in other words,
\[
  \int_{\RR^N} a_p(x)|\psi|^p\dx
  \le \|a_p\|_{L^\infty}\|\psi\|_{L^p}^p.
\]
This applies in particular to $r=\tau$ under \textnormal{(G3)}. Since
$\HH \hookrightarrow L^2(\RR^N;\CC^d)\cap L^p(\RR^N;\CC^d)$ continuously,
\eqref{eq:F0_growth}, \textnormal{(G3)}, and~\eqref{eq:holder_weighted}
imply that the functionals below are well-defined on the full space $\HH$.

The derivative term associated with $G$ is controlled similarly. For
$\psi,\varphi\in\HH$ and $p'=p/(p-1)$,
\[
  \int_{\RR^N} b_\tau |\psi|^{\tau-1}|\varphi|\,dx
  \le B_\tau \|\psi\|_{L^p}^{\tau-1}\|\varphi\|_{L^p},
  \qquad
  \frac1{p'}=\frac{p-\tau}{p}+\frac{\tau-1}{p}.
\]
Thus $g(\cdot,\psi)$ defines an element of
$L^{p'}(\RR^N;\CC^d)\subset H^{-1/2}(\RR^N;\CC^d)$. Similarly, by
\textnormal{(F2)}, $f_0(\cdot,\psi)$ belongs to
$L^2(\RR^N;\CC^d)+L^{p'}(\RR^N;\CC^d)\subset H^{-1/2}(\RR^N;\CC^d)$.
We set
\begin{equation}\label{eq:energy}
  E(\psi) \coloneqq \frac12 \norm{\psi^+}^2 - \frac12 \norm{\psi^-}^2 - \int_{\RR^N} F(x,\psi)\dx, \qquad \psi \in \HH,
\end{equation}
and decompose it as the even symmetric part plus the perturbation:
\begin{equation}\label{eq:E_decomp}
  E(\psi) = E_0(\psi) - \Phi(\psi),
\end{equation}
where
\begin{align}
  E_0(\psi) &\coloneqq \frac12 \norm{\psi^+}^2 - \frac12 \norm{\psi^-}^2 - \int_{\RR^N} F_0(x,\psi)\dx, \label{eq:E0}
  \\ \Phi(\psi) &\coloneqq \int_{\RR^N} G(x,\psi)\dx. \label{eq:Phi}
\end{align}
Standard Nemytskii arguments, together with the growth assumptions~\textnormal{(F2)} and~\textnormal{(G2)} and the preceding estimates, give $E,E_0,\Phi \in C^1(\HH,\RR)$ and show that, for all $\psi,\varphi \in \HH$,
the differentials are given by
\begin{align}
  \mathrm{d}E(\psi)[\varphi] &= (\psi^+,\varphi^+)_\HH - (\psi^-,\varphi^-)_\HH - \int_{\RR^N} \RE \inner{f(x,\psi)}{\varphi}_{\CC^d}\dx, \label{eq:dE}
  \\ \mathrm{d}E_0(\psi)[\varphi] & = (\psi^+,\varphi^+)_\HH - (\psi^-,\varphi^-)_\HH - \int_{\RR^N} \RE \inner{f_0(x,\psi)}{\varphi}_{\CC^d}\dx, \label{eq:dE0}
  \\ \mathrm{d}\Phi(\psi)[\varphi] &= \int_{\RR^N} \RE \inner{g(x,\psi)}{\varphi}_{\CC^d}\dx, \label{eq:dPhi}
\end{align}
where
\[
  f(x,\psi) \coloneqq f_0(x,\psi) + g(x,\psi) = \nabla_\psi^{\RR}F(x,\psi).
\]
We denote by $\nabla_\HH E(\psi)$ the $\HH$-gradient of $E$ at $\psi$,
that is, the unique element of $\HH$ such that, for every $\varphi \in \HH$,
\[
  \mathrm{d}E(\psi)[\varphi] = (\nabla_\HH E(\psi),\varphi)_\HH.
\]

\begin{remark}
If one wishes to write an explicit formula for the $\HH$-gradient, the correct Riesz representative of the nonlinear term is $\abs{\Dmt}^{-1}f(\cdot,\psi)$, not $\Dmt^{-1}f(\cdot,\psi)$. Formally,
\[
  \nabla_\HH E(\psi)
  =\psi^+-\psi^- - |\Dmt|^{-1}f(\cdot,\psi),
\]
where $f(\cdot,\psi)$ is viewed as an element of
$H^{-1/2}(\RR^N;\CC^d)$ and $|\Dmt|^{-1}:H^{-1/2}\to H^{1/2}$ is the
bounded inverse induced by the spectral gap. We do not use this explicit
expression in the sequel.
\end{remark}

We now prove a basic local regularity result for weak solutions of~\eqref{eq:main} obtained from the variational problem.

\begin{proposition} \label{prop:basic_local_regularity}
Assume $N=3$ and that the coefficients appearing in~\eqref{eq:main} are locally
bounded. Let $2<p<3$ and let $\psi\in H^{1/2}(\RR^3;\CC^4)$ be any weak solution obtained from the variational problem, either in the original formulation or in the renormalized formulation of~\cref{sec:renormalized_quadratic}. Then
\[
  \psi\in W^{1,r}_{\rm loc}(\RR^3;\CC^4), \qquad \text{where} \quad r \coloneqq \tfrac{3}{p-1}>\tfrac32.
\]
Hence, the solutions have one weak derivative locally in the natural elliptic integrability class.
\end{proposition}

\begin{proof}
The Sobolev embedding $H^{1/2}(\RR^3)\hookrightarrow L^3(\RR^3)$ gives
$\psi\in L^3_{\rm loc}$. By \textnormal{(F2)},
\[
 |f_0(x,\psi)|\le C\bigl(|\psi|+|\psi|^{p-1}\bigr) \implies f_0(\cdot,\psi)\in L^r_{\rm loc} \quad \text{with } r = \tfrac3{p-1}.
\]
Likewise, \textnormal{(G2)} gives $g(\cdot,\psi)\in L^{3/(\tau-1)}_{\rm loc}$; since
$\tau<p$, this space embeds into $L^r$ on bounded sets. In the renormalized formulation, the additional admissible bounded matrix
coefficient gives only a linear term in $\psi$. The bounded linear
coefficients, including $V$ and the admissible $B$, send
$\psi$ into $L^3_{\rm loc}\subset L^r_{\rm loc}$. Hence, after moving the
bounded linear terms to the left-hand side, the right-hand side of the
corresponding Dirac equation belongs to $L^r_{\rm loc}$. Let $\eta\in C_c^\infty(\RR^3)$ and set $w\coloneqq\eta\psi$. In the sense of distributions,
\[
  \Dm w=\eta\Dm\psi-\imath(\alpha\cdot\nabla\eta)\psi,
\]
where the second term belongs to $L^3\subset L^r$ on the support of $\eta$; hence $\Dm w\in L^r$, while $w\in L^r$ with compact support. The estimate for the constant-coefficient first-order elliptic operator
$\Dm=-\imath\alpha\cdot\nabla+m\beta$,
\[
\|u\|_{W^{1,r}} \le C\bigl(\|\Dm u\|_{L^r}+\|u\|_{L^r}\bigr),
\]
holds first for $u\in C_c^\infty(\RR^3;\CC^4)$ and follows from the
H\"ormander--Mikhlin multiplier theorem; see, e.g.,
\cite{TaylorPDE1}. We now extend it to $w$ by regularization. Indeed, for the mollifications $w_\delta=w*\rho_\delta$ one has $\Dm w_\delta=(\Dm w)*\rho_\delta$ because $\Dm$ has constant coefficients, so $w_\delta\to w$ and $\Dm w_\delta\to\Dm w$ in $L^r$. Applying the estimate to $w_\delta-w_{\delta'}$ shows that $(w_\delta)$ is a Cauchy sequence in $W^{1,r}$; its limit is $w$, so $\eta\psi\in W^{1,r}$ and the estimate holds for $\eta\psi$.
\end{proof}

% =====================================================================
\section{The symmetric functional}\label{sec:symmetric}
% =====================================================================

In this section we study the even functional $E_0$ on the invariant compactness space $X$ introduced in~\cref{sec:setting}, and prove that it possesses infinitely many critical values diverging to $+\infty$.

\subsection{The fountain geometry}\label{ssec:fountain}

Let $\{e_k^+\}_{k \geq 1}$ be an orthonormal basis of $X^+$. Define the corresponding finite-dimensional and tail subspaces
\[
  E_n^+ \coloneqq \mathrm{span} \bigl\{e_1^+, \ldots, e_n^+ \bigr\},
  \qquad
  Z_n \coloneqq \bigl(E_n^+\bigr)^\perp \cap X^+.
\]
Since $X = X^+ \oplus X^-$ by~\cref{prop:compact}, we have the orthogonal decomposition
\[
  X = E_n^+ \oplus Z_n \oplus X^-.
\]

\begin{lemma}[Vanishing of $L^p$-norms on tail subspaces]
\label{lem:vanishing}
Define
\[
  \beta_n \coloneqq \sup \left\{ \norm{\psi}_{L^p(\RR^N)} \: : \: \psi \in Z_n,\ \norm{\psi}=1 \right\}.
\]
Then $\beta_n \to 0$ as $n \to \infty$.
\end{lemma}

\begin{proof}
Since the embedding $X \hookrightarrow L^p(\RR^N;\CC^d)$ is continuous, the numbers $\beta_n$ are necessarily finite. Suppose, for contradiction, that there exist $\delta > 0$ and a sequence of functions $\psi_n \in Z_n$ such that, for all $n$,
\[
  \norm{\psi_n} = 1, \qquad \norm{\psi_n}_{L^p} \geq \delta.
\]
Since $(\psi_n)$ is bounded in $X$, up to a subsequence and relabeling, we may assume that $\psi_n \rightharpoonup \psi$ weakly in $X$. We claim that $\psi = 0$. Indeed, fix $k \geq 1$. For every $n \geq k$, since $\psi_n \in Z_n \subset Z_k$, $(\psi_n,e_k^+)_\HH = 0$. Since $X^+$ is a closed linear subspace of $X$, it is weakly closed; hence
the weak limit $\psi$ belongs to $X^+$. Passing to the limit gives
$(\psi,e_k^+)_\HH = 0$ for every $k$, hence 
\[
\psi \in X^+ \cap (X^+)^\perp = \{0\} \implies \psi_n \rightharpoonup 0 \quad \text{in } X.
\]
By the compact embedding $X \hookrightarrow L^p(\RR^N;\CC^d)$, the convergence is upgraded to
\[
  \psi_n \to 0 \qquad \text{strongly in } L^p(\RR^N;\CC^d),
\]
contradicting $\norm{\psi_n}_{L^p} \geq \delta$. Thus $\beta_n \to 0$.
\end{proof}

\begin{proposition}\label{prop:bn}
There are sequences $\rho_n > 0$ and $b_n \to +\infty$ such that
\[
  E_0(\psi) \geq b_n \qquad \text{for all } \psi \in Z_n \text{ with } \norm{\psi}=\rho_n.
\]
\end{proposition}

\begin{proof}
Let $\psi \in Z_n$ with $\norm{\psi} = \rho$. Since $Z_n \subset X^+$, we have $\psi^- = 0$, hence the even energy is given by
\[
  E_0(\psi) = \tfrac12 \rho^2 - \int_{\RR^N} F_0(x,\psi)\dx.
\]
By~\eqref{eq:F0_eps_growth}, for every $\varepsilon > 0$ there exists $C_\varepsilon > 0$ such that
\[
  \abs{F_0(x,\psi)} \leq \varepsilon \abs{\psi}^2 + C_\varepsilon \abs{\psi}^p.
\]
Therefore
\[
  E_0(\psi) \geq \tfrac12 \rho^2 - \varepsilon \norm{\psi}_{L^2}^2 - C_\varepsilon \norm{\psi}_{L^p}^p.
\]
Using the spectral gap estimate $\norm{\psi}_{L^2}^2 \leq \lambda_1^{-1}\norm{\psi}^2$ and the definition of $\beta_n$, we obtain
\[
  E_0(\psi) \geq \left(\tfrac12 - \tfrac{\varepsilon}{\lambda_1}\right)\rho^2 - C_\varepsilon \beta_n^p \rho^p.
\]
Choose $\varepsilon = \lambda_1/4$. Then
\[
  E_0(\psi) \ge \tfrac14 \rho^2 - C_\varepsilon \beta_n^p \rho^p.
\]
Since $X^+$ is infinite-dimensional, $Z_n\ne\{0\}$, and the embedding
into $L^p$ is injective; hence $\beta_n>0$. Set
\[
  \rho_n \coloneqq \left(\frac{1}{2p\,C_\varepsilon\,\beta_n^p}\right)^{1/(p-2)}.
\]
Then
\[
  C_\varepsilon\beta_n^p\rho_n^p
  = C_\varepsilon\beta_n^p\rho_n^2\rho_n^{p-2}
  = \tfrac{1}{2p}\rho_n^2.
\]
Consequently, whenever $\norm{\psi}=\rho_n$,
\[
  E_0(\psi)
  \ge \left(\tfrac14-\tfrac{1}{2p}\right)\rho_n^2
  = \tfrac{p-2}{4p}\left(\tfrac{1}{2p\,C_\varepsilon}\right)^{2/(p-2)}
  \beta_n^{-2p/(p-2)}
  \eqqcolon b_n.
\]
Since $\beta_n \to 0$ by~\cref{lem:vanishing}, we conclude that $b_n \to +\infty$.
\end{proof}

\subsection{Upper bound on slices with finite-dimensional positive part}

We now show that $E_0$ is non-positive on the large-radius part of $E_n^+\oplus X^-$.

\begin{proposition} \label{prop:an}
For each $n \geq 1$, there exists a radius $R_n > 0$ such that
\[
  E_0(\psi) \le 0 \qquad \text{for all } \psi \in E_n^+ \oplus X^- \text{ with } \norm{\psi} \geq R_n.
\]
\end{proposition}

\begin{proof}
We argue by contradiction. Assume that for some fixed $n$ there exists a sequence $\psi_j \in E_n^+ \oplus X^-$ such that
\[
  \norm{\psi_j} \to \infty, \qquad E_0(\psi_j) > 0.
\]
Decompose $\psi_j$ as $\psi_j = \psi_j^+ + \psi_j^-$ with $\psi_j^+ \in E_n^+$ and $\psi_j^- \in X^-$, and consider the unit-norm sequence
\[
  v_j \coloneqq \tfrac{\psi_j}{\norm{\psi_j}} = v_j^+ + v_j^-, \qquad \norm{v_j^+}^2 + \norm{v_j^-}^2 = \norm{v_j}^2 = 1.
\]
Since $E_n^+$ is finite-dimensional, up to a subsequence, $v_j^+ \to v^+$ strongly in $E_n^+$. Moreover, up to a further subsequence, $v_j^- \rightharpoonup v^-$ weakly in $X^-$. Hence, the sum also converges:
\[
  v_j \rightharpoonup v \coloneqq v^+ + v^- \qquad \text{weakly in } X.
\]
The compact embedding $X \hookrightarrow L^p(\RR^N;\CC^d)$ upgrades this to $v_j \to v$ strongly in $L^p(\RR^N;\CC^d)$ and, after passing to a subsequence, almost everywhere in $\RR^N$. We now distinguish two cases for the limit.

\smallskip
\textbf{Case 1: $v^+ \neq 0$.} Then $v \neq 0$, because otherwise $0 = v = v^+ + v^-$, which would imply $v^+ = -v^- \in X^+ \cap X^- = \{0\}$, a contradiction.
Therefore the nonzero locus of $v$,
\[
  \Omega \coloneqq \{x \in \RR^N\  : \ v(x) \neq 0\},
\]
has positive measure. For a.e.\ $x \in \Omega$,
\[
  \abs{\psi_j(x)}  = \norm{\psi_j}\abs{v_j(x)} \to +\infty.
\]
Using the global lower bound $F_0(x,\psi) \geq c_0 \abs{\psi}^p$ from~\textnormal{(F3)}, we have
\[
  \tfrac{F_0(x,\psi_j(x))}{\norm{\psi_j}^2} \geq c_0 \norm{\psi_j}^{p-2}\abs{v_j(x)}^p \to +\infty \qquad \text{for a.e.\ } x \in \Omega,
\]
since $p>2$ and $|v_j(x)|\to|v(x)|>0$ a.e.\ on $\Omega$. Since $F_0 \geq 0$, Fatou's lemma yields
\[
  \liminf_{j\to\infty} \tfrac{1}{\norm{\psi_j}^2} \int_{\RR^N} F_0(x,\psi_j)\dx = +\infty.
\]
Dividing the identity
\[
  E_0(\psi_j) = \tfrac12 \norm{\psi_j^+}^2 - \tfrac12 \norm{\psi_j^-}^2 - \int_{\RR^N} F_0(x,\psi_j)\dx
\]
by $\norm{\psi_j}^2$, we obtain
\[
  \tfrac{E_0(\psi_j)}{\norm{\psi_j}^2} = \tfrac12 \norm{v_j^+}^2 - \tfrac12 \norm{v_j^-}^2 - \tfrac{1}{\norm{\psi_j}^2} \int_{\RR^N} F_0(x,\psi_j)\dx \to -\infty,
\]
contradicting $E_0(\psi_j) > 0$.

\smallskip
\textbf{Case 2: $v^+ = 0$.} Then $\norm{v_j^+}\to 0$, hence $\norm{v_j^-}^2 = 1 - \norm{v_j^+}^2 \to 1$. Since $F_0 \geq 0$,
\[
  \tfrac{E_0(\psi_j)}{\norm{\psi_j}^2} \le \tfrac12 \norm{v_j^+}^2 - \tfrac12 \norm{v_j^-}^2  \to -\tfrac12,
\]
again contradicting $E_0(\psi_j) > 0$.
\end{proof}

\subsection{The Palais--Smale condition} 

First, we recall the notion of Palais--Smale condition, and then we prove that the functional satisfies it.

\begin{definition}[Palais--Smale condition]
We say that a functional $E_0$ satisfies the \emph{Palais--Smale (PS) condition on $X$ at the level $c \in \RR$} if any sequence $(u_n) \subset X$ such that
\[
  E_0(u_n) \to c, \qquad \mathrm{d}E_0(u_n) \to 0 \quad \text{in } X^*,
\]
admits a strongly converging subsequence in $X$.
\end{definition}

\begin{proposition}\label{prop:PS}
The functional $E_0$ satisfies the Palais--Smale condition on $X$ at every level $c \in \RR$.
\end{proposition}

\begin{proof}
We divide the proof into two steps: boundedness and strong convergence.

\smallskip
\noindent\textbf{Step 1.} We claim that $\norm{u_n}$ is bounded. Assume by contradiction that, up to a subsequence, $s_n \coloneqq \norm{u_n} \to +\infty$. Since $\mathrm dE_0(u_n)\to0$ in $X^*$, we have
$\mathrm dE_0(u_n)[u_n]=o(\|u_n\|)=o(s_n)$. Hence
\[
  E_0(u_n)-\tfrac12\mathrm dE_0(u_n)[u_n]
  =c+o(1)+o(s_n).
\]
On the other hand, for every $\psi\in X$,
\[
  E_0(\psi)-\tfrac12\mathrm dE_0(\psi)[\psi]
  =
  \int_{\RR^N}
  \left(
  \tfrac12\RE\inner{f_0(x,\psi)}{\psi}_{\CC^d}
  -F_0(x,\psi)
  \right)\dx.
\]
By \textnormal{(F3)},
\[
  \tfrac12 \RE \inner{f_0(x,\psi)}{\psi}_{\CC^d} - F_0(x,\psi)
  \ge \left(\tfrac p2 - 1\right) F_0(x,\psi) \ge 0.
\]
Therefore
\[
0\le
\left(\tfrac p2-1\right)\int_{\RR^N}F_0(x,u_n)\dx
\le |c|+o(1)+o(s_n),
\]
and, consequently, the integral is bounded above by
\begin{equation}\label{eq:F0_PS_bound}
  \int_{\RR^N} F_0(x,u_n)\dx \leq C(1+s_n).
\end{equation}
Since the global lower bound in \textnormal{(F3)} gives $F_0(x,\psi) \geq c_0\abs{\psi}^p$ for all $(x,\psi)$, we obtain directly
\begin{equation}\label{eq:Lp_PS_bound}
  \norm{u_n}_{L^p}^p \le \tfrac{1}{c_0}\int_{\RR^N} F_0(x,u_n)\dx \le C(1+s_n).
\end{equation}
Fix $\varepsilon_0 = \lambda_1/4$ and use~\eqref{eq:f0_eps_growth}: testing $\mathrm{d}E_0(u_n)$ against $u_n^+$ yields
\[
  \norm{u_n^+}^2 = \int_{\RR^N} \RE \inner{f_0(x,u_n)}{u_n^+}_{\CC^d}\dx + o(s_n),
\]
which ultimately gives
\[
  \norm{u_n^+}^2 \le \varepsilon_0 \norm{u_n}_{L^2}\norm{u_n^+}_{L^2} + C_{\varepsilon_0}\norm{u_n}_{L^p}^{p-1}\norm{u_n^+}_{L^p} + o(s_n).
\]
Substituting the following estimates into the above
\[
  \norm{u_n}_{L^2} \leq \lambda_1^{-1/2} s_n, \qquad
  \norm{u_n^+}_{L^2} \leq \lambda_1^{-1/2} s_n, \qquad
  \norm{u_n^+}_{L^p} \leq S_p s_n,
\]
we obtain
\[
  \norm{u_n^+}^2 \leq \tfrac14 s_n^2 + C_{\varepsilon_0} S_p \norm{u_n}_{L^p}^{p-1} s_n + o(s_n).
\]
The same estimate holds for $\norm{u_n^-}^2$, testing the energy $\mathrm{d}E_0(u_n)$ against $-u_n^-$. Indeed,
\[
 \mathrm dE_0(u_n)[-u_n^-]
 =\|u_n^-\|^2+\int_{\RR^N}
   \RE\inner{f_0(x,u_n)}{u_n^-}_{\CC^d}\dx=o(s_n),
\]
and the preceding estimates apply to the integral. Summing gives
\[
  s_n^2 \leq \tfrac12 s_n^2 + 2 C_{\varepsilon_0} S_p \norm{u_n}_{L^p}^{p-1} s_n + o(s_n).
\]
By~\eqref{eq:Lp_PS_bound}, we have $\norm{u_n}_{L^p}^{p-1} \leq C(1+s_n)^{(p-1)/p}$; hence,
\[
  \tfrac12 s_n^2 \le C s_n \bigl(1+s_n\bigr)^{(p-1)/p} + o(s_n).
\]
Since
\[
  1 + \tfrac{p-1}{p} = \tfrac{2p-1}{p} = 2 - \tfrac{1}{p} < 2 \qquad \text{for all } p > 1,
\]
the RHS is $o(s_n^2)$ as $s_n \to \infty$, a contradiction. Hence $(u_n)$ is bounded in $X$.

\smallskip
\noindent\textbf{Step 2.}
Since $(u_n)$ is bounded in $X$, up to a subsequence, it converges weakly to some $u$ in $X$, which upgrades to
\[
  u_n \to u \qquad \text{strongly in } L^p(\RR^N;\CC^d)
\]
by the compact embedding $X \hookrightarrow L^p(\RR^N;\CC^d)$. Set $w_n \coloneqq u_n - u$. Then $w_n \rightharpoonup 0$ in $X$. Since $P_m$ and $Q_m$ are bounded on $X$, also the components converge weakly:
\[
  w_n^\pm \rightharpoonup 0 \qquad \text{in } X.
\]
By compactness, both convergences are strong in $L^p(\RR^N;\CC^d)$. We now show $\norm{w_n^+}\to 0$. Since $\mathrm{d}E_0(u_n)\to 0$ in $X^*$ and $(w_n^+)$ is bounded in $X$, $\mathrm{d}E_0(u_n)[w_n^+] = o(1)$. Expanding this identity gives
\[
  o(1) = \norm{w_n^+}^2 + (u^+,w_n^+)_\HH - \int_{\RR^N} \RE \inner{f_0(x,u_n)}{w_n^+}_{\CC^d}\dx.
\]
Since $w_n^+ \rightharpoonup 0$ in $X$, we have $(u^+,w_n^+)_\HH \to 0$. Now fix $\varepsilon > 0$ and use~\eqref{eq:f0_eps_growth}:
\[
  \abs{f_0(x,u_n)} \leq \varepsilon \abs{u_n} + C_\varepsilon \abs{u_n}^{p-1}.
\]
Hence
\[
  \left| \int_{\RR^N} \RE \inner{f_0(x,u_n)}{w_n^+}_{\CC^d}\dx \right| \le \varepsilon \norm{u_n}_{L^2}\norm{w_n^+}_{L^2} + C_\varepsilon \norm{u_n}_{L^p}^{p-1}\norm{w_n^+}_{L^p} \le C \varepsilon \norm{w_n^+} + o(1),
\]
because $(u_n)$ is bounded in $X$, hence in $L^2$ and $L^p$, and $w_n^+ \to 0$ in $L^p$. It follows that
\[
  \norm{w_n^+}^2 \leq C \varepsilon \norm{w_n^+} + o(1).
\]
Since $(w_n^+)$ is bounded in $X$, taking the limsup gives
\[
  \limsup_{n\to\infty}\norm{w_n^+} \le C\varepsilon.
\]
Letting $\varepsilon \downarrow 0$, we obtain
\[
  \norm{w_n^+} \to 0.
\]
The proof that $\norm{w_n^-}\to 0$ is identical, testing $\mathrm{d}E_0(u_n)$ against $-w_n^-$. We conclude that
\[
  \norm{w_n}^2 =  \norm{w_n^+}^2 + \norm{w_n^-}^2 \to 0.
\]
Thus $u_n \to u$ strongly in $X$.
\end{proof}

\subsection{Infinitely many critical values}

We now prove that the restricted even functional has infinitely many distinct
critical points. When its full Hilbert gradient preserves $X$, these critical
points are weak solutions of the unrestricted equation.

\begin{theorem}\label{thm:symmetric}
Under assumptions \textnormal{(V1)}--\textnormal{(V3)},
\textnormal{(X1)}--\textnormal{(X2)}, and
\textnormal{(F0)}--\textnormal{(F4)}, the functional $E_0$ on $X$
possesses a sequence of diverging critical values $c_n^0 \to +\infty$, with
corresponding critical points $\psi_n^0 \in X$. If, in addition,
$\nabla_\HH E_0(X)\subset X$, then every $\psi_n^0$ is a critical point on
$\HH$ and hence a weak solution of the full symmetric equation.
\end{theorem}

\begin{proof}
We apply the generalized fountain theorem for strongly indefinite functionals
on $X=X^+\oplus X^-$, in the form of
\cite[Theorem~4.8]{BartschDing2006}.
The functional $E_0$ is even by \textnormal{(F0)}. Moreover, writing
\[
E_0(\psi) = \tfrac12\|\psi^+\|^2 - \tfrac12\|\psi^-\|^2 -
\Psi(\psi),
\qquad
\Psi(\psi)\coloneqq \int_{\RR^N}F_0(x,\psi)\dx,
\]
the map $\Psi$ is of class $C^1$ on $X$ by~\cref{sec:assumptions}. Fix an orthonormal basis $(e_j)_{j\ge1}$ of $X^-$. We use the
Kryszewski--Szulkin topology $\tau$ on $X=X^+\oplus X^-$ induced by the norm
\[
\|\psi\|_\tau \coloneqq \|\psi^+\| + \sum_{j=1}^\infty 2^{-j}
\bigl|(\psi^-,e_j)_\HH\bigr|.
\]
Thus $\psi_n\to\psi$ in $\tau$ is equivalent to saying that $\psi_n^+\to\psi^+$ strongly in
$X^+$ and
\[
(\psi_n^-,e_j)_\HH\longrightarrow (\psi^-,e_j)_\HH
\qquad\text{for every }j\ge1.
\]
In particular, if $(\psi_n^-)$ is bounded in $X^-$, this coordinate-wise
convergence implies $\psi_n^-\rightharpoonup\psi^-$ in $X^-$. We verify explicitly the two continuity hypotheses used by the strongly
indefinite fountain theorem.

First, suppose that $\psi_n\to\psi$ in the $\tau$-topology and that
$E_0(\psi_n)\ge a$ for some $a\in\RR$. Then, as mentioned above, $\psi_n^+\to\psi^+$ strongly in
$X^+$. Since $F_0\ge0$ by \textnormal{(F3)}, we have the estimate
\[
a\le E_0(\psi_n)
\le
\tfrac12\|\psi_n^+\|^2-\tfrac12\|\psi_n^-\|^2,
\]
so $(\psi_n^-)$ is bounded. Hence $\psi_n^-\rightharpoonup\psi^-$ in $X^-$.
Consequently $\psi_n\rightharpoonup\psi$ in $X$, and compactness gives
$\psi_n\to\psi$ in $L^p$. The sequence $(\psi_n)$ is bounded in $X$, hence in
$L^2\cap L^p$. By the mean-value formula and \eqref{eq:f0_eps_growth}, for every
$\varepsilon>0$,
\[
|\Psi(\psi_n)-\Psi(\psi)| \le \varepsilon
\bigl(\|\psi_n\|_{L^2}+\|\psi\|_{L^2}\bigr)
\|\psi_n-\psi\|_{L^2}  + C_\varepsilon
\bigl(\|\psi_n\|_{L^p}^{p-1}+\|\psi\|_{L^p}^{p-1}\bigr)
\|\psi_n-\psi\|_{L^p}.
\]
The second term tends to zero. For the first term we use only boundedness of
$(\psi_n)$ and $(\psi_n-\psi)$ in $L^2$, obtaining
\[
  \limsup_{n\to\infty}|\Psi(\psi_n)-\Psi(\psi)|\le C\varepsilon.
\]
Letting $\varepsilon\downarrow0$ gives $\Psi(\psi_n)\to\Psi(\psi)$. Finally,
weak lower semicontinuity of the norm on $X^-$ gives
\[
\limsup_{n\to\infty}E_0(\psi_n)
\le
\tfrac12\|\psi^+\|^2-\tfrac12\|\psi^-\|^2-
\lim_{n\to\infty}\Psi(\psi_n)
=E_0(\psi),
\]
which is the required $\tau$-upper semicontinuity.

Second, let $\psi_n\rightharpoonup\psi$ in $X$. Again
$\psi_n\to\psi$ in $L^p$. We prove that the nonlinear part of the gradient is
weakly sequentially continuous. Fix $\varphi\in X$ and choose
$\varphi_m\in L^\infty_c(\RR^N;\CC^d)$ such that
\[
\|\varphi_m-\varphi\|_{L^2}
+
\|\varphi_m-\varphi\|_{L^p}
\longrightarrow0.
\]
For fixed $m$, the strong $L^p$ convergence implies convergence in measure on
$\operatorname{supp}\varphi_m$. By the Carath\'eodory property,
$f_0(x,\psi_n(x))\to f_0(x,\psi(x))$ in measure on this support. Moreover,
\eqref{eq:f0_eps_growth} with $\varepsilon=1$ gives
\[
|f_0(x,z)|\le C\bigl(|z|+|z|^{p-1}\bigr),
\]
and the boundedness of $(\psi_n)$ in $L^2\cap L^p$ implies uniform
integrability on $\operatorname{supp}\varphi_m$. Hence Vitali's theorem gives
\[
\int_{\RR^N}
\RE\inner{f_0(x,\psi_n)-f_0(x,\psi)}{\varphi_m}_{\CC^d}\dx
\longrightarrow0.
\]
The error made by replacing $\varphi_m$ with $\varphi$ is bounded uniformly in
$n$ by
\[
C\bigl(\|\varphi-\varphi_m\|_{L^2}
+\|\varphi-\varphi_m\|_{L^p}\bigr),
\]
using the same growth bound and boundedness of $(\psi_n)$ in $L^2\cap L^p$.
Letting first $n\to\infty$ and then $m\to\infty$ gives
\[
\int_{\RR^N}
\RE\inner{f_0(x,\psi_n)-f_0(x,\psi)}{\varphi}_{\CC^d}\dx
\longrightarrow0.
\]
Thus the nonlinear part of the gradient is weakly sequentially continuous.
Together with weak continuity of the linear spectral part, this proves
\[
\nabla_XE_0(\psi_n)\rightharpoonup\nabla_XE_0(\psi)
\qquad\text{in }X.
\]
The geometric assumptions are provided by~\cref{prop:bn,prop:an}: after
choosing the tail radius in \cref{prop:bn}, the slice radius in
\cref{prop:an} may be enlarged so that it is strictly larger. The
PS condition is given by \cref{prop:PS}. Therefore the generalized
fountain theorem applies and yields a sequence of critical values
$c_n^0$ diverging to infinity:
\[
c_n^0\ge b_n\longrightarrow+\infty.
\]
If $\nabla_\HH E_0(X)\subset X$, restricted criticality implies that
$\nabla_\HH E_0(\psi_n^0)$ is both orthogonal to $X$ and contained in $X$;
hence it vanishes, proving the final assertion.
\end{proof}

% =====================================================================
\section{Reduction to the positive spectral space}\label{sec:perturbation_symmetry}
% =====================================================================

The perturbation-from-symmetry argument is carried out after eliminating the
negative spectral variable. This avoids a strongly indefinite minimax
construction for the nonsymmetric functional. The price is an explicit
semiconvexity hypothesis on the perturbing primitive, namely
\textnormal{(G5)}. Convex perturbing primitives correspond to the case
$\mu\equiv0$ and are therefore covered without an amplitude restriction.

Throughout this section we assume \textnormal{(V1)}--\textnormal{(V3)},
\textnormal{(X1)}--\textnormal{(X2)}, and \textnormal{(G1)}--\textnormal{(G5)}.
We also use the compact embedding $X\hookrightarrow L^p(\RR^3;\CC^4)$.
The final deformation argument requires $\tau \in (1,p/2)$, while the fiber
concavity and anti-coercivity estimates below only use $\tau \in (1,2)$.
Additionally, the homogeneous part is
\begin{equation}\label{eq:exact_homogeneous_model}
 F_0(x,z)=\tfrac{b(x)}p|z|^p,
 \qquad b\in C_b(\RR^3;\RR), \quad 0<b_0\le b(x)\le b_1, \quad 2<p<3,
\end{equation}
and we write
\[
 P_0(u)\coloneqq\tfrac1p\int_{\RR^3}b(x)|u|^p\,\mathrm dx.
\]
Thus $P_0(u)=\int_{\RR^3}F_0(x,u)\,\mathrm dx$. We define the energy path by
\[
 E_\vartheta(u) \coloneqq E_0(u)-\vartheta \, \Phi(u),
 \qquad \vartheta\in[0,1].
\]
We regard the complex Hilbert spaces as real Hilbert spaces by taking the real
part of the Hermitian product.

\subsection{The negative-fiber maximizer}

Recall that $\lambda_1$ denotes the spectral gap of the operator $\Dmt$ as stated in \textnormal{(V3)}. Set
\[
 \kappa \coloneqq \tfrac{\|\mu\|_{L^\infty}}{\lambda_1}<1,
\]
where $\mu$ is the function in \textnormal{(G5)}. For $x\in X^+$ and
$v\in X^-$ define the energy path along the direction $v$ as
\[
 \mathfrak f_{\vartheta,x}(v) \coloneqq E_\vartheta(x+v).
\]

\begin{lemma}[Strong concavity of the negative fibers]
\label{lem:fibre_strong_concavity}
For every $\vartheta\in[0,1]$ and $x\in X^+$,
$\mathfrak f_{\vartheta,x}:X^-\to\RR$ is coercive from above and uniformly
strongly concave. More precisely, it satisfies
\begin{equation}\label{eq:fibre_strong_monotonicity}
 \bigl\langle D_v\mathfrak f_{\vartheta,x}(v)
       -D_v\mathfrak f_{\vartheta,x}(w),v-w\bigr\rangle
 \le -(1-\kappa)\|v-w\|^2, \qquad \text{for all } v,w\in X^-.
\end{equation}
Consequently, there is a unique maximizer
$h_\vartheta(x)\in X^-$, characterized by the identity
\begin{equation}\label{eq:fibre_stationarity}
 E_\vartheta'(x+h_\vartheta(x))[\eta]=0
 \qquad\text{for every }\eta\in X^-.
\end{equation}
In addition, the map $(\vartheta,x)\mapsto h_\vartheta(x)$ is continuous from
$[0,1]\times X^+$ to $X^-$.
\end{lemma}

\begin{proof}
Let $\xi\in X^+$ be fixed and let $v,w\in X^-$. For simplicity, we denote the relevant scalar product by
\[
(\star) \coloneqq \bigl\langle D_v\mathfrak f_{\vartheta,\xi}(v)
       -D_v\mathfrak f_{\vartheta,\xi}(w),v-w\bigr\rangle.
\]
Set $A_p(z)=|z|^{p-2}z$. 
Since the negative
quadratic part of $E_\vartheta$ is $-\|v\|^2/2$ on the fiber, we have
\begin{align*}
 (\star) & = -\|v-w\|^2 -\vartheta\bigl\langle
 \Phi'(\xi+v)-\Phi'(\xi+w),v-w\bigr\rangle \\
&\qquad -\int_{\RR^3} b(y)\,\RE\left\langle
 A_p\bigl(\xi(y)+v(y)\bigr)-A_p\bigl(\xi(y)+w(y)\bigr),
 v(y)-w(y)\right\rangle\dy.
\end{align*}
The map $z\mapsto |z|^p/p$ is convex, and $b\ge b_0>0$; hence the integral
in the second line is nonnegative. Moreover, \textnormal{(G5)}, the spectral
gap condition, and $\|u\|^2\ge\lambda_1\|u\|_{L^2}^2$ give
\[
 \bigl\langle \Phi'(\xi+v)-\Phi'(\xi+w),v-w\bigr\rangle
 \ge-\|\mu\|_\infty\|v-w\|_{L^2}^2
 \ge-\kappa\|v-w\|^2.
\]
Hence, putting everything together yields the inequality
\[
(\star)
 \le -(1-\vartheta\kappa)\|v-w\|^2
 \le -(1-\kappa)\|v-w\|^2,
\]
which proves \eqref{eq:fibre_strong_monotonicity}.
By the lower-order estimate on $\Phi$ and the continuous embedding
$X\hookrightarrow L^p$, we have
\[
 \mathfrak f_{\vartheta,\xi}(v)
 \le \tfrac12\|\xi\|^2-\tfrac12\|v\|^2+C\|\xi+v\|^\tau \implies \mathfrak f_{\vartheta,\xi}(v)\xrightarrow{\|v\|\to\infty} -\infty,
\]
uniformly
for $\vartheta\in[0,1]$ and $\xi$ in bounded sets. Any maximizing sequence is
therefore bounded. If $v_j\rightharpoonup v$ in $X^-$, then
$\xi+v_j\rightharpoonup \xi+v$ in $X$, and the compact embedding gives
$\xi+v_j\to\xi+v$ strongly in $L^p$. Hence
\[
 P_0(\xi+v_j)\to P_0(\xi+v),\qquad
 \Phi(\xi+v_j)\to\Phi(\xi+v),
\]
where the second convergence follows from \textnormal{(G3)} and weighted
H\"older. Since $-\|v\|^2/2$ is weakly upper semicontinuous, the direct
method gives a maximizer, and \eqref{eq:fibre_strong_monotonicity} gives both
uniqueness and \eqref{eq:fibre_stationarity}.

To prove the last assertion, consider the double sequence $(\vartheta_j,x_j)\to(\vartheta,x)$ and set
$v_j \coloneqq h_{\vartheta_j}(x_j)$. Coercivity gives boundedness of $(v_j)$. Let
$v=h_\vartheta(x)$. Since
\[
D_v\mathfrak f_{\vartheta_j,x_j}(v_j)=0 = D_v\mathfrak f_{\vartheta,x}(v),
\]
using \eqref{eq:fibre_strong_monotonicity} at the parameter
$(\vartheta_j,x_j)$ gives
\[
(1-\kappa)\|v_j-v\|^2
\le -\bigl\langle D_v\mathfrak f_{\vartheta_j,x_j}(v_j)
       -D_v\mathfrak f_{\vartheta_j,x_j}(v),v_j-v\bigr\rangle =\bigl\langle D_v\mathfrak f_{\vartheta_j,x_j}(v)
       -D_v\mathfrak f_{\vartheta,x}(v),v_j-v\bigr\rangle.
\]
Consequently, unless $v_j=v$,
\[
 (1-\kappa)\|v_j-v\|
 \le
 \left\|D_v\mathfrak f_{\vartheta_j,x_j}(v)
       -D_v\mathfrak f_{\vartheta,x}(v)\right\|_{(X^-)^*}.
\]
The last dual norm tends to zero because $x_j\to x$ strongly in $X^+$,
$\vartheta_j\to\vartheta$, $v$ is fixed, and the weighted Nemytskii maps
$u\mapsto b|u|^{p-2}u$ and $u\mapsto g(\cdot,u)$ are continuous into the
corresponding dual spaces. Thus $v_j\to v$ in $X^-$.
\end{proof}

\begin{definition}
The \emph{reduced path on the positive spectral space} is defined via the function $h_\vartheta$ as follows:
\begin{equation}\label{eq:reduced_path}
 J_\vartheta(x)=E_\vartheta(x+h_\vartheta(x)),
 \qquad x\in X^+, \quad \vartheta\in[0,1].
\end{equation}
\end{definition}

\begin{proposition}[Critical-point correspondence]
\label{prop:reduced_path}
The map $J:[0,1]\times X^+\to\RR$ belongs to
$C^1([0,1]\times X^+,\RR)$, with derivatives understood up to the endpoints
of $[0,1]$, and satisfies
\begin{align}
 J_\vartheta'(x)[\xi]
 &=E_\vartheta'(x+h_\vartheta(x))[\xi],
 \qquad \xi\in X^+,\label{eq:reduced_derivative}\\
 \partial_\vartheta J_\vartheta(x)
 &=-\Phi(x+h_\vartheta(x)).\label{eq:reduced_velocity}
\end{align}
Furthermore,
\[
 J_\vartheta'(x)=0
 \iff
 E_\vartheta'(x+h_\vartheta(x))[\varphi]=0
 \quad\text{for every }\varphi\in X.
\]
Thus critical points of $J_\vartheta$ correspond to critical points of the
restriction $E_\vartheta|_X$. If, for this value of $\vartheta$, one also has
\[
  \nabla_\HH E_\vartheta(X)\subset X,
\]
then the corresponding restricted critical points are full critical points on
$\HH$. In particular, for $\vartheta=1$, this is exactly the assumption
\textnormal{(X3)}.

In addition, the family $J$ satisfies the \emph{joint Palais--Smale condition}: if
$\vartheta_j\to\vartheta$, $J_{\vartheta_j}(x_j)$ is bounded, and
$\|J_{\vartheta_j}'(x_j)\|\to0$, then $(x_j)$ has a convergent subsequence in
$X^+$.
\end{proposition}

\begin{proof}
Continuity follows from \cref{lem:fibre_strong_concavity}. We first prove
Fr\'echet differentiability with respect to $x$. For $y\in X^+$, the
maximizing property gives
\begin{align*}
 E_\vartheta(x+y+h_\vartheta(x))-E_\vartheta(x+h_\vartheta(x))
 &\le J_\vartheta(x+y)-J_\vartheta(x)\\
 &\le E_\vartheta(x+y+h_\vartheta(x+y))
      -E_\vartheta(x+h_\vartheta(x+y)).
\end{align*}
The lower bound equals
\[
 E_\vartheta'(x+h_\vartheta(x))[y]+o(\|y\|).
\]
For the upper bound, the fundamental theorem of calculus gives
\[
 \int_0^1
 E_\vartheta'\bigl(x+h_\vartheta(x+y)+sy\bigr)[y]\,\mathrm ds.
\]
Since $h_\vartheta(x+y)\to h_\vartheta(x)$ as $y\to0$ and
$E_\vartheta'$ is continuous, this expression is also
\[
 E_\vartheta'(x+h_\vartheta(x))[y]+o(\|y\|).
\]
The two-sided estimate proves the Fr\'echet derivative formula
\eqref{eq:reduced_derivative}. The same comparison for an increment in
$\vartheta$, with one-sided increments at the endpoints of $[0,1]$, gives
\eqref{eq:reduced_velocity}. Since $(\vartheta,x)\mapsto x+h_\vartheta(x)$ is
continuous, the Nemytskii maps in $E_\vartheta'$ are continuous, and $\Phi$ is
continuous, the right-hand sides of
\eqref{eq:reduced_derivative}--\eqref{eq:reduced_velocity} depend
continuously on $(\vartheta,x)$. Thus the two partial derivatives are
jointly continuous, and the standard product-space criterion yields
$J\in C^1([0,1]\times X^+,\RR)$. The critical-point correspondence follows
from \eqref{eq:reduced_derivative} and the negative stationarity condition
\eqref{eq:fibre_stationarity}.

\smallskip
We now prove the joint Palais--Smale condition in detail. Set
\[
 u_j=x_j+h_{\vartheta_j}(x_j),
 \qquad
 \varepsilon_j=\|J_{\vartheta_j}'(x_j)\|.
\]
By \eqref{eq:fibre_stationarity}, the restriction of
$E_{\vartheta_j}'(u_j)$ to $X^-$ vanishes; by
\eqref{eq:reduced_derivative}, its restriction to $X^+$ is
$J_{\vartheta_j}'(x_j)$. Since the splitting is orthogonal, this immediately implies that
\begin{equation}\label{eq:joint_PS_full_derivative}
  \|E_{\vartheta_j}'(u_j)\|_{X^*}=\varepsilon_j\longrightarrow0.
\end{equation}
Exact $p$-homogeneity gives
\begin{equation}\label{eq:joint_PS_energy_identity}
 E_{\vartheta_j}(u_j)-\tfrac12E_{\vartheta_j}'(u_j)[u_j]
 =\left(\tfrac p2-1\right)P_0(u_j)
  +\vartheta_j\left(\tfrac12\Phi'(u_j)[u_j]-\Phi(u_j)\right).
\end{equation}
Here $E_{\vartheta_j}(u_j)=J_{\vartheta_j}(x_j)$, so the energies are bounded.
Moreover, \eqref{eq:joint_PS_full_derivative} gives
$E_{\vartheta_j}'(u_j)[u_j]=O(\varepsilon_j\|u_j\|)$. By \textnormal{(G4)}
and weighted H\"older,
\[
 \left|\tfrac12\Phi'(u_j)[u_j]-\Phi(u_j)\right|
 \le C\|u_j\|_{L^p}^{\tau},
\]
and $P_0(u_j)\ge (b_0/p)\|u_j\|_{L^p}^p$. Hence
\[
 \left(\tfrac p2-1\right)P_0(u_j)
 \le C+C\varepsilon_j\|u_j\|+C P_0(u_j)^{\tau/p}.
\]
Since $\tau/p<1$, Young's inequality absorbs the final term into the left-hand
side. Therefore
\begin{equation}\label{eq:joint_PS_P0_bound}
  P_0(u_j)\le C\bigl(1+\varepsilon_j\|u_j\|\bigr).
\end{equation}
Testing $E_{\vartheta_j}'(u_j)$ with $u_j^+$ and with $-u_j^-$, respectively,
and using \textnormal{(G2)}, weighted H\"older, and the continuous
$X\hookrightarrow L^p$ embedding, gives bounds on the norm of both components:
\begin{align*}
 \|u_j^+\|
 &\le \varepsilon_j
  +C\Bigl(P_0(u_j)^{(p-1)/p}
          +P_0(u_j)^{(\tau-1)/p}\Bigr),\\
 \|u_j^-\|
 &\le \varepsilon_j
  +C\Bigl(P_0(u_j)^{(p-1)/p}
          +P_0(u_j)^{(\tau-1)/p}\Bigr).
\end{align*}
For example,
\[
 \|b_\tau|u_j|^{\tau-1}\|_{L^{p'}}
 \le B_\tau\|u_j\|_{L^p}^{\tau-1},
 \qquad
 \|b|u_j|^{p-1}\|_{L^{p'}}
 \le C P_0(u_j)^{(p-1)/p}.
\]
Combining these estimates with \eqref{eq:joint_PS_P0_bound}, and using
$\varepsilon_j\le1$ for large $j$, yields
\[
  \|u_j\|\le C\bigl(1+\|u_j\|^{(p-1)/p}\bigr).
\]
Because $(p-1)/p<1$, the sequence $(u_j)$ is bounded in $X$. After taking a subsequence,
\[
 u_j\rightharpoonup u\quad\text{in }X,
 \qquad
 u_j\to u\quad\text{in }L^p(\RR^3;\CC^4)\text{ and almost everywhere.}
\]
The exact power map satisfies $b|u_j|^{p-2}u_j\longrightarrow b|u|^{p-2}u$ in $L^{p'}$. 
We also have
\begin{equation}\label{eq:weighted_Nemytskii_convergence}
 g(\cdot,u_j)\longrightarrow g(\cdot,u)
 \quad\text{in }L^{p'}(\RR^3;\CC^4).
\end{equation}
For completeness, set $p'=p/(p-1)$ and choose
\[
 K_m=B_m(0)\cap \bigl\{ x \ : \ b_\tau(x)\le m \bigr\}.
\]
Then $K_m$ has finite measure, $b_\tau$ is bounded on $K_m$, and dominated
convergence gives
\[
 \|b_\tau\mathbf1_{\RR^3\setminus K_m}\|_{L^{p/(p-\tau)}}
 \longrightarrow0.
\]
On $K_m$, the Carath\'eodory property and a.e.\ convergence give
$g(x,u_j(x))\to g(x,u(x))$ almost everywhere. Since $b_\tau$ is bounded on $K_m$ and
$(\tau-1)p'<p$, the family $|g(\cdot,u_j)|^{p'}$ is uniformly integrable on
$K_m$; hence Vitali's theorem gives convergence in $L^{p'}(K_m)$. On the
complement, the growth bound and H\"older's inequality give, uniformly in $j$,
\[
 \bigl\|(g(\cdot,u_j)-g(\cdot,u))
       \mathbf1_{\RR^3\setminus K_m} \bigr\|_{L^{p'}}
 \le C \bigl\|b_\tau\mathbf1_{\RR^3\setminus K_m}\bigr\|_{L^{p/(p-\tau)}}
       \bigl(\|u_j\|_{L^p}^{\tau-1}+\|u\|_{L^p}^{\tau-1}\bigr),
\]
which proves \eqref{eq:weighted_Nemytskii_convergence}. Since
$\vartheta_j\to\vartheta$, it also follows that
\[
 \vartheta_jg(\cdot,u_j)\to\vartheta g(\cdot,u)
 \quad\text{in }L^{p'}(\RR^3;\CC^4).
\]
Passing to the limit in \eqref{eq:joint_PS_full_derivative} now gives
$E_\vartheta'(u)=0$. Set $w_j=u_j-u$. Since
$E_{\vartheta_j}'(u_j)\to0$ in $X^*$, $E_\vartheta'(u)=0$, and
$(w_j^\pm)$ is bounded in $X$, we have
\[
 \bigl(E_{\vartheta_j}'(u_j)-E_\vartheta'(u)\bigr)[w_j^\pm]=o(1).
\]
Subtracting the two equations and testing first with $w_j^+$ gives
\begin{align*}
 \|w_j^+\|^2
 &=\bigl(E_{\vartheta_j}'(u_j)-E_\vartheta'(u)\bigr)[w_j^+] \\
 &\quad+\int_{\RR^3}\RE\left\langle
 b\bigl(|u_j|^{p-2}u_j-|u|^{p-2}u\bigr)
 +\vartheta_jg(x,u_j)-\vartheta g(x,u),w_j^+
 \right\rangle\,\mathrm dx=o(1).
\end{align*}
The first term is $o(1)$ by the preceding estimate. The integral is $o(1)$
by the $L^{p'}$ convergences above and boundedness of $w_j^+$ in $L^p$.
Testing with $-w_j^-$ gives in the same way $\|w_j^-\|^2=o(1)$. Hence
$u_j\to u$ strongly in $X$. Finally, $x_j=P_m u_j\to P_m u$ in $X^+$,
which proves the joint Palais--Smale condition.
\end{proof}

\begin{remark}\label{rem:J0_even}
Since $E_0$ is even and the negative-fiber maximizer is unique,
$h_0(-x)=-h_0(x)$. Hence $J_0(-x)=J_0(x)$; that is, the initial reduced
functional is also even.
\end{remark}

\subsection{Velocity bounds and finite-dimensional anti-coercivity}

In this section we establish the velocity bounds and show that the reduced path is uniformly anti-coercive, with respect to $\vartheta$, on every finite-dimensional subspace of $X^+$.

\begin{lemma}[Bolle velocity bounds]
\label{lem:reduced_velocity_bounds}
Set $\alpha \coloneqq \tau/p$.  For every $b>0$ there is $C_b>0$ such that
\begin{equation}\label{eq:reduced_H2}
|J_\vartheta(x)|\le b \implies |\partial_\vartheta J_\vartheta(x)|
 \le C_b\bigl(\|J_\vartheta'(x)\|+1\bigr) \bigl(\|x\|+1 \bigr).
\end{equation}
Additionally, at every critical point of $J_\vartheta$,
\begin{equation}\label{eq:reduced_H3}
 |\partial_\vartheta J_\vartheta(x)|
 \le C\bigl(1+|J_\vartheta(x)|^\alpha\bigr).
\end{equation}
\end{lemma}

\begin{proof}
Set $u=x+h_\vartheta(x)$ and $D=J_\vartheta'(x)$. Since the negative
component of $E_\vartheta'(u)$ vanishes,
$E_\vartheta'(u)[u]=D[x]$. Exact $p$-homogeneity yields the identity
\begin{equation}\label{eq:reduced_radial_identity}
 J_\vartheta(x)-\tfrac12D[x]
 =\left(\tfrac p2-1\right)P_0(u)
 +\vartheta\left(\tfrac12\Phi'(u)[u]-\Phi(u)\right).
\end{equation}
By \textnormal{(G3)}--\textnormal{(G4)} and
$P_0(u)\ge c\|u\|_{L^p}^p$, it implies that
\[
 |\Phi(u)|+
 \left|\tfrac12\Phi'(u)[u]-\Phi(u)\right|
 \le C P_0(u)^\alpha.
\]
If $|J_\vartheta(x)|\le b$, \eqref{eq:reduced_radial_identity} and Young's
inequality imply
\[
 P_0(u)\le C_b\bigl(1+\|D\| \|x\|\bigr).
\]
Together with \eqref{eq:reduced_velocity}, this gives
\[
 |\partial_\vartheta J_\vartheta(x)|
 \le C_b\bigl(1+\|D\|\|x\|\bigr)^\alpha.
\]
Since $0<\alpha<1$,
\[
 (1+AB)^\alpha\le C(1+A)(1+B),\qquad A,B\ge0,
\]
and \eqref{eq:reduced_H2} follows. If $D=0$, the same identity gives
$P_0(u)\le C(1+|J_\vartheta(x)|)$, and
\eqref{eq:reduced_H3} follows.
\end{proof}

\begin{lemma}[Uniform anti-coercivity]
\label{lem:reduced_anti_coercivity}
For every finite-dimensional subspace $W\subset X^+$,
\[
 \sup_{\vartheta\in[0,1]}J_\vartheta(x)\longrightarrow-\infty
 \qquad\text{as }\|x\|\to\infty,
 \quad x\in W.
\]
\end{lemma}

\begin{proof}
Because $X^+\perp X^-$, one has $\|x+v\|\ge\|x\|$ for
$x\in W$ and $v\in X^-$. It is therefore enough to prove
\begin{equation}\label{eq:full_slice_uniform_anticoercivity}
 \sup_{\vartheta\in[0,1]}E_\vartheta(u)\longrightarrow-\infty
 \qquad\text{as }\|u\|\to\infty,
 \quad u\in W\oplus X^-.
\end{equation}
Suppose that \eqref{eq:full_slice_uniform_anticoercivity} fails. Then there
are $\vartheta_j\in[0,1]$ and
$u_j=x_j+v_j\in W\oplus X^-$ such that
\[
 s_j\coloneqq\|u_j\|\longrightarrow\infty,
 \qquad
 E_{\vartheta_j}(u_j)\ge-C.
\]
Set $y_j=u_j/s_j=a_j+b_j$, where $a_j=x_j/s_j\in W$ and
$b_j=v_j/s_j\in X^-$. After passing to a subsequence,
\[
 a_j\to a\quad\text{strongly in }W,
 \qquad
 b_j\rightharpoonup b\quad\text{weakly in }X^-,
\]
and hence, by compactness of $X\hookrightarrow L^p$, the sum converges:
\begin{equation}\label{eq:normalized_slice_Lp_limit}
  y_j\to y\coloneqq a+b\qquad\text{strongly in }L^p.
\end{equation}
In addition, the perturbation is negligible at the quadratic scale:
\begin{equation}\label{eq:lower_order_quadratic_negligible}
 \tfrac{|\Phi(u_j)|}{s_j^2}
 \le C\tfrac{\|u_j\|_{L^p}^{\tau}}{s_j^2}
 \le Cs_j^{\tau-2}\longrightarrow0.
\end{equation}
If $a=0$, then $\|a_j\|\to0$ and, because
$\|a_j\|^2+\|b_j\|^2=1$, one has $\|b_j\|^2\to1$. Since $P_0\ge0$,
\[
 \limsup_{j\to\infty}\tfrac{E_{\vartheta_j}(u_j)}{s_j^2}
 \le\lim_{j\to\infty}
 \left(\tfrac12\|a_j\|^2-\tfrac12\|b_j\|^2\right)=-\tfrac12,
\]
where \eqref{eq:lower_order_quadratic_negligible} was used. This contradicts
$E_{\vartheta_j}(u_j)/s_j^2\ge-C/s_j^2$.

Assume now that $a\ne0$.  We claim that $y\ne0$ in $L^p$. Indeed, if
$y=0$ a.e., then $a=-b$ as distributions. Since both sides belong to $X$,
this is equality in $X$; hence
$a\in X^+\cap X^-=\{0\}$, a contradiction. Thus
$\|y\|_{L^p}>0$. By \eqref{eq:normalized_slice_Lp_limit} and
$0<b_0\le b(x)$,
\[
 \int_{\RR^3}b(x)|y_j|^p\,\mathrm dx
 \longrightarrow
 \int_{\RR^3}b(x)|y|^p\,\mathrm dx
 \ge b_0\|y\|_{L^p}^p>0.
\]
Consequently,
\[
 \tfrac{P_0(u_j)}{s_j^2}
 =\tfrac{s_j^{p-2}}p\int_{\RR^3}b(x)|y_j|^p\,\mathrm dx
 \longrightarrow+\infty.
\]
The quadratic quotient is bounded and the perturbation quotient tends to
zero by \eqref{eq:lower_order_quadratic_negligible}; therefore
$E_{\vartheta_j}(u_j)/s_j^2\to-\infty$, again a contradiction. This proves
\eqref{eq:full_slice_uniform_anticoercivity}. Finally,
$J_\vartheta(x)=\max_{v\in X^-}E_\vartheta(x+v)$ gives the asserted
anti-coercivity.
\end{proof}

\subsection{Polynomial deformation from symmetry}

We use the polynomial deformation-from-symmetry framework developed in
\cite{Bolle1999,BolleGhoussoubTehrani2000,ChambersGhoussoub2001}, in the
$C^1$ formulation stated by
\citet[Theorem~3.1]{SalvatoreSquassina2003}. We only need the special case in
which the finite-dimensional initial space in the abstract theorem is
$\{0\}$.

\begin{theorem}[Polynomial deformation from symmetry]
\label{thm:polynomial_deformation}
Let $H$ be a separable real Hilbert space and let
$J:[0,1]\times H\to\RR$ be a $C^1$ family. Suppose that the following conditions hold:
\begin{enumerate}[(i)]
 \item the family $J$ satisfies the joint Palais--Smale condition;
 \item the functional $J_\vartheta$ is even at $\vartheta=0$;
 \item the local velocity estimate analogous to \eqref{eq:reduced_H2} holds;
 \item at critical points,
 \[
   |\partial_\vartheta J_\vartheta(x)|
   \le C(1+|J_\vartheta(x)|^\alpha),
   \qquad 0\le\alpha<1;
 \]
 \item $J_\vartheta$ is uniformly anti-coercive on every finite-dimensional
 subspace of $H$.
\end{enumerate}
Let $H_k=\operatorname{span}\{e_1,\ldots,e_k\}$ for an orthonormal basis of
$H$, let
\[
 \mathscr H = \bigl\{\gamma\in C(H,H) \ : \ \gamma\text{ is odd and }
 \gamma(u)=u\text{ for }\|u\|\text{ sufficiently large} \bigr\},
\]
and define the critical (Bolle) levels
\begin{equation}\label{eq:bolle_levels}
 c_k \coloneqq \inf_{\gamma\in\mathscr H}\sup_{u\in\gamma(H_k)}J_0(u).
\end{equation}
If, for some $\gamma_0>1/(1-\alpha)$, they satisfy
\begin{equation}\label{eq:bolle_levels_assumption}
 c_k\ge C_1k^{\gamma_0}-C_2,
\end{equation}
then $J_1$ has a sequence of critical values tending to $+\infty$.
\end{theorem}

\begin{remark}
\label{rem:velocity_fields}
After increasing the constant if necessary, the critical velocity in
hypothesis~(iv) is controlled by
\[
 \varrho_1(\vartheta,s) \coloneqq -C_*(1+s^2)^{\alpha/2},
 \qquad
 \varrho_2(\vartheta,s) \coloneqq C_*(1+s^2)^{\alpha/2}.
\]
They are continuous and globally Lipschitz in $s$ because $0\le\alpha<1$,
so they are admissible velocity fields in the Chambers--Ghoussoub--Bolle
theorem. Their scalar flows are
$\psi_i(1,s)=s+O(1+s^\alpha)$ as $s\to+\infty$, so the deformation can
shift a level of size $s$ by order at most $s^\alpha$ over the unit parameter
interval. Mere divergence of the initial levels is therefore not sufficient;
the quantitative growth condition in the theorem makes the initial minimax
levels dominate this distortion and yields endpoint critical values unbounded
above.
\end{remark}

It remains to verify the critical-level assumption
\eqref{eq:bolle_levels_assumption}. This is done in
\cref{prop:reduced_level_growth} with $\gamma_0=2$, after the following
technical result.

\begin{lemma}
\label{lem:reduced_intersection_all_radii}
With the notations of~\cref{thm:polynomial_deformation}, for every $k\ge1$, every $\gamma\in\mathscr H$, and every $\rho>0$,
\[
 \gamma(H_k)\cap Z_{k-1}\cap\{x\in H:\|x\|=\rho\}\ne\varnothing,
 \qquad Z_{k-1}=H_{k-1}^{\perp}.
\]
\end{lemma}

\begin{proof}
Take $R_\gamma>0$ so that $\gamma(u)=u$ whenever
$\|u\|\ge R_\gamma$. If $\rho\ge R_\gamma$, then
$\gamma(\rho e_k)=\rho e_k$, and $e_k\in Z_{k-1}$. If
$0<\rho<R_\gamma$, restrict $\gamma$ to the ball of radius $R_\gamma$, i.e.
\[
 D_{k,\gamma} = \bigl\{u\in H_k \ : \ \|u\|\le R_\gamma \bigr\}.
\]
This restriction is odd and equals the identity on
$\partial D_{k,\gamma}$. The standard sphere-intersection lemma for odd maps
that equal the identity on the boundary of a finite-dimensional ball,
\citep[Proposition~9.23]{Rabinowitz1986}, applied with the finite-dimensional
space $H_k$, the complementary subspace $Z_{k-1}$, the boundary radius
$R_\gamma$, and the target radius $\rho<R_\gamma$, gives a point
$u\in D_{k,\gamma}$ such that
\[
  \gamma(u)\in Z_{k-1},\qquad \|\gamma(u)\|=\rho.
\]
This concludes the proof.
\end{proof}

\begin{proposition}
\label{prop:reduced_level_growth}
Let $X=X_{\rm pw}$, let $2<p<3$, and choose the positive basis supplied by
\cref{lem:nested_tail_basis}. The levels \eqref{eq:bolle_levels} of the
reduced symmetric functional $J_0$ satisfy
\begin{equation}\label{eq:reduced_quadratic_growth}
 c_k\ge C_1k^2-C_2.
\end{equation}
\end{proposition}

\begin{proof}
By \cref{lem:finite_angular_width_transfer,lem:nested_tail_basis} and
\citet[Theorem~1 and Remark~1]{DotaSkrzypczak2025}, applied to the radial
fractional Sobolev embedding
\[
  RH^{1/2}_2(\RR^3)\hookrightarrow RL^p(\RR^3),
\]
with parameters $m = n =1$, $\gamma_1=d=3$, $s_1=1/2$, $s_2 = 0$, $p_1=2$, and $p_2=p$, where $m$ and $n$ denote the block-counting parameters of the cited theorem
(the number of blocks and the number of blocks of least dimension), unrelated
to the mass in $\Dm$, the basis can be chosen so that
\begin{equation}\label{eq:reduced_beta_rate}
 \beta_k\coloneqq
 \sup_{\substack{x\in Z_k\\\|x\|=1}}\|x\|_{L^p}
 \le Ck^{-(p-2)/p}\qquad(k\ge1),
\end{equation}
where $Z_k=H_k^\perp$. Set also
\[
  \beta_0=\sup_{\|x\|=1}\|x\|_{L^p}.
\]
Indeed, the compactness parameter in that theorem is
\[
 \delta=s_1-s_2-d\left(\tfrac1{p_1}-\tfrac1{p_2}\right)
 =\tfrac12-3\left(\tfrac12-\tfrac1p\right)=\tfrac3p-1>0,
\]
which is precisely the range $p<3$. Since $2 < p$, the relevant case is
$2\le p_1<p_2\le\infty$. With
$1/\mathfrak p \coloneqq 1/p_1-1/p_2=1/2-1/p$, and with $n=1$ eliminating the logarithmic
factor, the exponent is
\[
 \tfrac{\gamma_1-1}{\mathfrak p}
 =2\left(\tfrac12-\tfrac1p\right)=\tfrac{p-2}{p}.
\]
By \cref{lem:reduced_intersection_all_radii}, for every
$\gamma\in\mathscr H$ and every $\rho>0$ there is
\[
 x_{\gamma,\rho}\in
 \gamma(H_k)\cap \bigl\{x\in Z_{k-1} \ : \ \|x\|=\rho \bigr\}.
\]
Since $h_0(x)$ maximizes the negative fiber,
\[
 J_0(x)\ge E_0(x) \ge \tfrac12\|x\|^2-C\|x\|_{L^p}^p.
\]
For $x=x_{\gamma,\rho}\in Z_{k-1}$, the definition of $\beta_{k-1}$ gives
$\|x\|_{L^p}\le \beta_{k-1}\rho$. Hence, for every $\gamma\in\mathscr H$ and
$\rho>0$,
\[
 \sup_{u\in\gamma(H_k)}J_0(u)
 \ge \tfrac12\rho^2-C\beta_{k-1}^p\rho^p.
\]
Since $\gamma$ was arbitrary, we may choose
\[
  \rho_k=
  \left(\frac{1}{2pC\beta_{k-1}^p}\right)^{1/(p-2)}.
\]
Then
\[
  C\beta_{k-1}^p\rho_k^p
  = C\beta_{k-1}^p\rho_k^2\rho_k^{p-2}
  =\frac1{2p}\rho_k^2,
\]
and therefore
\[
 \tfrac12\rho_k^2-C\beta_{k-1}^p\rho_k^p
 =\left(\tfrac12-\tfrac1{2p}\right)\rho_k^2
 =\frac{p-1}{2p}\rho_k^2.
\]
Consequently,
\[
 c_k\ge C\beta_{k-1}^{-2p/(p-2)}-C.
\]
For $k\ge2$, using \eqref{eq:reduced_beta_rate} gives
$c_k\ge C_1k^2-C_2$ after changing constants.  Enlarging $C_2$ covers
$k=1$ as well and proves \eqref{eq:reduced_quadratic_growth}. Finiteness of $c_k$ follows from \cref{lem:reduced_anti_coercivity} by taking
$\gamma$ equal to the identity.
\end{proof}

% =====================================================================

\section{Lower-order non-even partial-wave theorem}\label{sec:lower_order_reduced}
% =====================================================================

We can now prove the first main result: the existence of infinitely many
critical points for semiconvex lower-order perturbations.

\begin{theorem}[Existence] \label{thm:partialwave_lower_order}
Let $N=3$, $d=4$, $2<p<3$, and set $X=X_{\rm pw}$, where $X_{\rm pw}$ is the lowest partial-wave space with radial coefficients. Assume \textnormal{(V1)}--\textnormal{(V3)}, and assume that this space $X_{\rm pw}$ satisfies \textnormal{(X1)}--\textnormal{(X2)}. Suppose
\[
 F_0(x,\psi)=\tfrac{b(|x|)}p|\psi|^p, 
 \qquad b\in C_b([0,\infty);\RR), \quad 0<b_0\le b(r)\le b_1<\infty.
\]
Let $G$ satisfy \textnormal{(G1)}--\textnormal{(G5)} with exponent $\tau \in (1,p/2)$. Then the restricted functional $E|_{X_{\rm pw}}$ has a sequence of critical
points $u_k\in X_{\rm pw}$ such that
\[
 E(u_k)\to+\infty.
\]
If, in addition, the full perturbed energy satisfies the channel-invariance
condition \textnormal{(X3)}, i.e.
\[
  \nabla_\HH E(X_{\rm pw})\subset X_{\rm pw},
\]
then the spinors $u_k$ are weak solutions of the full Dirac equation~\eqref{eq:main}. 
\end{theorem}

\begin{remark}
The non-even perturbations in \cref{ex:Xpw_model_perturbation} give concrete
examples satisfying this additional invariance condition \textnormal{(X3)}.
For instance, under the radial hypotheses in \cref{cor:concrete_partial_wave_model}, the assumptions \textnormal{(X1)}--\textnormal{(X2)} on $X_{\rm pw}$ are satisfied.
\end{remark}

\begin{proof}
Set $X=X_{\rm pw}$. The exact-power choice of $F_0$ satisfies
\textnormal{(F0)}--\textnormal{(F4)} and is precisely the homogeneous model
used in \cref{sec:perturbation_symmetry}. By
\cref{lem:fibre_strong_concavity,prop:reduced_path}, critical points of the
reduced endpoint $J_1$ are in one-to-one correspondence with critical points
of the restriction $E|_{X_{\rm pw}}$.

The reduced path satisfies the joint PS condition by \cref{prop:reduced_path},
the local and critical velocity estimates by \cref{lem:reduced_velocity_bounds},
and finite-dimensional anti-coercivity by \cref{lem:reduced_anti_coercivity}.
Its initial functional is even by \cref{rem:J0_even}. We use on
$X_{\rm pw}^+$ the positive orthonormal basis supplied by
\cref{lem:nested_tail_basis}. With this choice,
\cref{prop:reduced_level_growth} gives quadratic initial levels. The critical velocity exponent is $\alpha=\tau/p$. Since
\[
 2>\tfrac1{1-\tau/p}
 \iff
 \tau<\tfrac p2,
\]
\cref{thm:polynomial_deformation} gives an unbounded sequence of critical
values of $J_1$. Let $x_k$ be corresponding critical points and set
\[
  u_k=x_k+h_1(x_k).
\]
Then $u_k\in X_{\rm pw}$, $u_k$ is critical for $E|_{X_{\rm pw}}$, and
\[
  E(u_k)=J_1(x_k)\to+\infty.
\]
If \textnormal{(X3)} holds, then $\nabla_\HH E(u_k)\in X_{\rm pw}$. Since
$u_k$ is critical for $E|_{X_{\rm pw}}$, we also have
$\nabla_\HH E(u_k)\perp_\HH X_{\rm pw}$. Hence
$\nabla_\HH E(u_k)=0$, so $u_k$ is a weak solution of
\eqref{eq:main}.
\end{proof}

% =====================================================================

\section{Renormalized even-quadratic perturbations}\label{sec:renormalized_quadratic}
% =====================================================================

A Hermitian operator-linear quadratic term is even in the spinor and can be
absorbed into the Dirac operator. Thus a quadratic contribution
\[
 -\tfrac12\int_{\RR^3}\RE\langle B(x)\psi,\psi\rangle \dx
\]
in the energy is incorporated into the linear part. The preceding reduction
then applies to the renormalized spectral splitting, provided the semiconvexity
bound is measured against the renormalized gap.

\subsection{Renormalizing an even quadratic perturbation}

Let $B:\RR^3\to\mathcal L(\CC^4)$ be a bounded measurable Hermitian matrix
field. As in \cref{sec:partialwaves}, write $\mathcal H_{\rm pw}\coloneqq\mathcal H_{-1,1/2}$ so that the lowest partial-wave subspace is given by
\[
X_{\rm pw}=H^{1/2}(\RR^3;\CC^4)\cap\mathcal H_{\rm pw}.
\]
We use the same notation $\PiPW$ for the bounded extensions of the angular
projection onto $\mathcal H_{\rm pw}$ to $H^{1/2}$ and $H^{-1/2}$.

\begin{definition}[Admissible operator] \label{def:admissible}
We say that $B$ is \emph{admissible on the lowest partial-wave space $X_{\rm pw}$} if
multiplication by $B$ leaves the underlying $L^2$ channel
$\mathcal H_{\rm pw}$ invariant and the operator
\begin{equation}\label{eq:renormalized_DB}
 \mathcal D_B=-\imath\alpha\cdot\nabla+m\beta+V_{\rm rad}(|x|)I_4-B(x)
\end{equation}
reduces $\mathcal H_{\rm pw}$. Its restriction to $\mathcal H_{\rm pw}$ is
required to have form domain $X_{\rm pw}$ with norm equivalent to the
$H^{1/2}$-norm, to satisfy the spectral gap condition
\begin{equation}\label{eq:DB_gap}
 \lambda_B \coloneqq \operatorname{dist}\bigl(0,\sigma(\mathcal D_B|_{\mathcal H_{\rm pw}})\bigr)>0,
\end{equation}
and to have infinite-dimensional positive and negative spectral subspaces.
\end{definition}

For a given admissible $B$, we consider the decomposition of $X_{\rm pw}$ and induced $\| \cdot \|_B$-norm corresponding to the renormalized operator~\eqref{eq:renormalized_DB} as
\[
 X_{\rm pw}=X_B^+\oplus X_B^-,
 \qquad
 \|u\|_B^2=\||\mathcal D_B|^{1/2}u\|_{L^2}^2.
\]

\begin{proposition}
\label{prop:concrete_admissible_B}
Assume that the radially symmetric potential $V_{\rm rad}$ satisfies
\[
V_{\rm rad} \in L^\infty\bigl((0,\infty);\RR\bigr),
 \qquad
 \lim_{R\to\infty}\|V_{\rm rad}\|_{L^\infty((R,\infty))}=0.
\]
Let $W_0,W_1\in L^\infty\bigl((0,\infty);\RR\bigr)$ be compactly supported radial
functions and set
\[
 B(x)=W_0(|x|)I_4+W_1(|x|)\beta.
\]
If the spectral gap condition holds, i.e.
\[
 0\notin\sigma\bigl((\mathcal D_m+V_{\rm rad}(|x|)I_4-B(x))
       |_{\mathcal H_{\rm pw}}\bigr),
\]
then $B$ is admissible on $X_{\rm pw}$.
\end{proposition}

\begin{proof}
For simplicity, set
\[
 A = (\mathcal D_m+V_{\rm rad}(|x|)I_4)|_{\mathcal H_{\rm pw}},
 \qquad
 A_B=A-B|_{\mathcal H_{\rm pw}}.
\]
Both $I_4$ and $\beta$ preserve the two fixed angular components of the
lowest channel, so multiplication by $B$ leaves $\mathcal H_{\rm pw}$
invariant and the full operator $\mathcal D_B$ reduces this $L^2$ channel.

\smallskip
\noindent\emph{Operator and form domains.}
The Clifford relations~\eqref{eq:clifford_relations} give, for $u\in H^1(\RR^3;\CC^4)$,
\[
 \|\mathcal D_m u\|_{L^2}^2
 =\|\nabla u\|_{L^2}^2+m^2\|u\|_{L^2}^2.
\]
Hence the free Dirac graph norm is equivalent to the $H^1$-norm, and the
bounded potential operator $V_{\rm rad} I_4$ does not change this fact. Thus, the domain of the operator $A$ is
\[
 \dom (A)=H^1(\RR^3;\CC^4)\cap\mathcal H_{\rm pw}
\]
with equivalent graph and $H^1$ norms. Since $B$ is bounded and Hermitian,
$A_B$ is self-adjoint on the same operator domain. Moreover, the two
graph norms are equivalent: there exists $C>0$ such that
\[
 C^{-1} \bigl(\|u\|+\|A_Bu\| \bigr) \le \|u\|+\|Au\|\le C \bigl(\|u\|+\|A_Bu\|\bigr).
\]
For a self-adjoint operator $T$, the spectral theorem identifies
$\bigl(L^2,\mathcal D(T)\bigr)$ with a couple of weighted $L^2$ spaces having
weights $1$ and $1+\lambda^2$, respectively. Complex interpolation of these weights
gives $(1+\lambda^2)^{1/2}$, which is equivalent to the real-valued
function $1+|\lambda|$. Consequently, one has
\[
 \dom \bigl(|T|^{1/2}\bigr) = \bigl[L^2,\dom (T)\bigr]_{1/2}
\]
with equivalence of the natural norms; see also
\citet[Chapter~V]{Kato1995}. Applying this identity to $A$ and $A_B$ gives
\[
 \dom \bigl(|A_B|^{1/2}\bigr)
 =\dom \bigl(|A|^{1/2}\bigr)
 =H^{1/2}(\RR^3;\CC^4)\cap\mathcal H_{\rm pw}
 =X_{\rm pw},
\]
and the corresponding form graph norms are equivalent. Once the spectral
gap for $A_B$ is established below, the homogeneous norm
$\||A_B|^{1/2}u\|_{L^2}$ is equivalent to the full form graph norm.

\smallskip
\noindent\emph{Relative compactness.}
Choose a ball $K$ containing the essential support of $B$. To prove that
$B(A-\imath)^{-1}$ is compact, let $(f_j)$ be bounded in
$\mathcal H_{\rm pw}$ and set
\[
u_j=(A-\imath)^{-1}f_j.
\]
Since $(A-\imath)^{-1}$ maps $\mathcal H_{\rm pw}$ continuously into
$\dom (A)=H^1(\RR^3;\CC^4)\cap\mathcal H_{\rm pw}$, the sequence
$(u_j)$ is bounded in $H^1(\RR^3;\CC^4)$. Hence the restrictions
$(u_j|_K)$ are bounded in $H^1(K;\CC^4)$, and the Rellich--Kondrachov theorem
(see, for instance, \citet[Theorem~6.3]{AdamsFournier2003}) yields a
subsequence converging strongly in $L^2(K;\CC^4)$. Since $B\in L^\infty$ is
supported in $K$, multiplication by $B$ is bounded on $L^2(K)$; hence
$(Bu_j)$ converges strongly in $L^2$. Thus
\[
B(A-\imath)^{-1}:\mathcal H_{\rm pw}
\longrightarrow\mathcal H_{\rm pw}
\]
is compact. The resolvent identity
\[
(A_B-\imath)^{-1}-(A-\imath)^{-1}
=(A_B-\imath)^{-1}B(A-\imath)^{-1}
\]
then shows that the resolvent difference is compact. By the essential $L^\infty$-tail condition on the potential $V_{\rm rad}$, the multiplication
operator $V_{\rm rad} I_4$ is relatively compact with respect to the free radial Dirac
operator by the same local Rellich and tail argument. Weyl's essential spectrum theorem (see, e.g.,
\citet[Theorem~XIII.14]{ReedSimon1978}) therefore gives
\[
\sigma_{\rm ess}(A_B)=\sigma_{\rm ess}(A)
=(-\infty,-m]\cup[m,\infty).
\]
Since $0\notin\sigma(A_B)$ and $\sigma(A_B)$ is closed,
\[
\lambda_B:=\operatorname{dist}(0,\sigma(A_B))>0.
\]
Thus the homogeneous form norm is equivalent to the $H^{1/2}$-norm, as
claimed above. Finally, the two nonempty unbounded components of the essential
spectrum imply that the positive and negative spectral projections of $A_B$
both have infinite rank. All requirements in the definition of admissibility
(\cref{def:admissible}) are therefore satisfied.
\end{proof}

\begin{remark}\label{rem:simple_B_gap}
In~\cref{prop:concrete_admissible_B}, the hypothesis is automatic if
\[
  \|V_{\rm rad}(|\cdot|)I_4-B\|_{L^\infty(\RR^3; \, \mathcal L(\CC^4))}<m.
\]
Indeed, with $W=V_{\rm rad} I_4-B$, one has
\[
  \mathcal D_B = \bigl(I+W\mathcal D_m^{-1}\bigr)\mathcal D_m,
  \qquad
  \|W\mathcal D_m^{-1}\|_{L^2\to L^2}\le \|W\|_\infty/m<1.
\]
Thus $\mathcal D_B$ is invertible on the full $L^2$ space, and therefore its
restriction to the reducing partial-wave channel $\mathcal H_{\rm pw}$ is
invertible as well.
\end{remark}

We now have all the ingredients to prove the main result of this section, i.e. the existence of infinitely many critical points under certain assumptions.

\begin{theorem} \label{thm:renormalized_even_quadratic}
Let $2<p<3$ and
\[
 F_0(x,\psi)=\tfrac{b(|x|)}p|\psi|^p,
 \qquad b\in C_b([0,\infty);\RR), \quad 0<b_0\le b(r)\le b_1<\infty.
\]
Let $B$ be admissible on $X_{\rm pw}$ (\cref{def:admissible}) and let
$H:\RR^3\times\CC^4\to\RR$ be a Carath\'eodory primitive, $C^1$ in the spinor
variable, with $H(x,0)=0$, satisfying the analogues of
\textnormal{(G1)}--\textnormal{(G4)} with $1<\tau<p/2$. Assume additionally
that, for some non-negative $\mu_B\in L^\infty(\RR^3)$,
\[
 \RE\langle h(x,z)-h(x,w),z-w\rangle
 \ge-\mu_B(x)|z-w|^2,
 \qquad
 \|\mu_B\|_\infty<\lambda_B,
\]
where $h=\nabla_\psi^{\RR}H$ and $\lambda_B$ is given by~\eqref{eq:DB_gap}.
Let $\PiPW$ be the angular projection
onto $\mathcal H_{\rm pw}$ and assume that, for every
$\psi\in X_{\rm pw}$,
\begin{equation}\label{eq:renormalized_channel_condition}
 \PiPW\bigl(b(|x|)|\psi|^{p-2}\psi+h(x,\psi)\bigr)
 =b(|x|)|\psi|^{p-2}\psi+h(x,\psi)
 \qquad\text{in }H^{-1/2}(\RR^3;\CC^4).
\end{equation}
For $\vartheta\in[0,1]$, define the renormalized path of energies
\begin{equation}\label{eq:renormalized_energy_path}
 E_{B,\vartheta}(u)
 =
 \tfrac12\|u_B^+\|_B^2-\tfrac12\|u_B^-\|_B^2
 -\int_{\RR^3}\tfrac{b(|x|)}p|u|^p\dx
 -\vartheta\int_{\RR^3}H(x,u)\dx.
\end{equation}
The endpoint $E_{B,1}$ is the renormalized energy associated with
$\mathcal D_B$.
Then
\begin{equation}\label{eq:renormalized_equation}
 \mathcal D_B\psi=b(|x|)|\psi|^{p-2}\psi+h(x,\psi)
\end{equation}
has infinitely many weak solutions $\psi_k\in X_{\rm pw}$ whose renormalized
energy values satisfy $E_{B,1}(\psi_k)\to+\infty$.
\end{theorem}

\begin{proof}
Let $h_{B,\vartheta}:X_B^+\to X_B^-$ be the negative-fiber maximizer for
$E_{B,\vartheta}$ and set
\[
  J_{B,\vartheta}(x)=E_{B,\vartheta}\bigl(x+h_{B,\vartheta}(x)\bigr).
\]
The renormalized splitting and norm have exactly the structural properties
used in \cref{sec:perturbation_symmetry}. The estimates of that section apply
verbatim to $E_{B,\vartheta}$, with $\lambda_1$ replaced by $\lambda_B$ and
$\|\cdot\|$ replaced by $\|\cdot\|_B$. The semiconvexity estimate
gives strong concavity of the renormalized negative fibers with
constant $1-\|\mu_B\|_\infty/\lambda_B>0$, and the reduced path satisfies
the joint PS condition, the local and critical velocity bounds, and
finite-dimensional anti-coercivity. At $\vartheta=0$, the functional
$E_{B,0}$ is even; uniqueness of the negative-fiber maximizer gives
$h_{B,0}(-x)=-h_{B,0}(x)$, so $J_{B,0}$ is even.

\smallskip
Choose the Hilbert basis $(e_{B,j}^+)$ of $X_B^+$ using
\cref{lem:nested_tail_basis} applied to the restricted embedding
$X_B^+\hookrightarrow L^p$. Set
\[
  Z_{B,k}=\overline{\operatorname{span}}\bigl\{e_{B,j}^+ \ : \ j\ge k+1 \bigr\},
  \qquad
  \beta_{B,k}=\sup_{\substack{u\in Z_{B,k}\\\|u\|_B=1}}\|u\|_{L^p}.
\]
Since $\|\cdot\|_B$ and the standard $H^{1/2}$-norm are equivalent on
$X_{\rm pw}$, the approximation numbers of $X_{\rm pw}\hookrightarrow L^p$
change only by a multiplicative constant. Restricting this embedding to the
closed subspace $X_B^+$ can only decrease approximation numbers, so the chosen
positive tail basis satisfies
\[
 \beta_{B,k}\le Ck^{-(p-2)/p}.
\]
Thus the reduced symmetric levels have quadratic growth. The velocity
exponent is $\tau/p<1/2$, and
\cref{thm:polynomial_deformation} applies to the renormalized reduced path.
The critical-point correspondence gives critical points of the energy
restricted to $X_{\rm pw}$ with $E_{B,1}$-values tending to $+\infty$. 
Let $u$ be one of these restricted critical points and set
\[
 R \coloneqq \mathcal D_B u - b(|x|)|u|^{p-2}u - h(x,u) \in H^{-1/2}.
\]
Restricted criticality gives $\langle R,\phi\rangle=0$ for every
$\phi\in X_{\rm pw}$. Since $\mathcal D_B$ reduces the underlying $L^2$
channel, its form operator commutes with $\PiPW$ between $H^{1/2}$ and
$H^{-1/2}$. Together with \eqref{eq:renormalized_channel_condition}, this
gives $\PiPW R=R$ in $H^{-1/2}$. Moreover, $\PiPW$ is self-adjoint for the
$H^{-1/2}$--$H^{1/2}$ duality pairing. Hence, for every
$\varphi\in H^{1/2}(\RR^3;\CC^4)$,
\[
  \langle R,\varphi\rangle
  =\langle \PiPW R,\varphi\rangle
  =\langle R,\PiPW\varphi\rangle.
\]
Since $\PiPW\varphi\in X_{\rm pw}$, restricted criticality gives
$\langle R,\PiPW\varphi\rangle=0$. Hence $R=0$, and the critical points are
full weak solutions of \eqref{eq:renormalized_equation}.
\end{proof}

\begin{remark}
\label{rem:concrete_renormalized_perturbations}
The theorem has fully concrete instances. For instance, let $B$ be of the form
\[
  B(x) = W_0(|x|)I_4 + W_1(|x|)\beta
\]
with $W_0,W_1\in L^\infty(0,\infty;\RR)$ compactly supported, and assume either the
resolvent condition in \cref{prop:concrete_admissible_B} or the smallness
condition in \cref{rem:simple_B_gap}. Then $B$ is
admissible (in the sense of~\cref{def:admissible}). For the lower-order primitive one may take exactly the
channel-preserving construction of \cref{ex:Xpw_model_perturbation}, namely
\[
  H(x,\psi)=
  \begin{cases}
   \varepsilon a(|x|)
   K_{\tau,\zeta} \left(\RE\inner{\eta(x/|x|)}{\psi}_{\CC^4}\right),&x\ne0,\\
   0,&x=0,
  \end{cases}
\]
where $a\in C_c^\infty((0,\infty))$, $a\ge0$, $\varepsilon>0$, $\zeta>0$,
$\zeta\ne1$, and $K_{\tau,\zeta}$ and $\eta$ are as in
\cref{ex:Xpw_model_perturbation}. This $H$ is non-even, satisfies the
renormalized analogues of \textnormal{(G1)}--\textnormal{(G4)}, satisfies the
semiconvexity condition with $\mu_B\equiv0$, and obeys
\eqref{eq:renormalized_channel_condition}.
\end{remark}

\begin{remark}
Condition \textnormal{(G5)} is automatic for convex perturbing primitives and
then imposes no restriction on their amplitude. A perturbation with a
concave component is also allowed, but its negative curvature must remain
strictly below the spectral gap in the sense of \textnormal{(G5)}. Without
such a fiber condition, uniqueness of the negative-fiber maximizer can fail,
and the positive-space reduction used here is no longer available.
\end{remark}

\bibliographystyle{plainnat}
\bibliography{cas-refs}

\end{document}